\documentclass[12pt]{amsart}
\usepackage[margin=1in,includeheadfoot]{geometry}
\usepackage{amsmath}
\usepackage{slashed}
\usepackage{graphicx}

\usepackage[margin=1in,includeheadfoot]{geometry}
\usepackage{amsmath}
\usepackage{slashed}
\usepackage{graphicx}
\usepackage{mathrsfs}
\usepackage{txfonts}
\usepackage{amssymb,dsfont}
\usepackage{enumerate}
\usepackage{slashed}
\usepackage{fancyhdr}
\usepackage{aliascnt}
\usepackage{pdfsync}
\usepackage{hyperref}

\newtheorem{thm}{Theorem}[section]

\newtheorem{lem}[thm]{Lemma}
\newtheorem{prop}[thm]{Proposition}
\theoremstyle{definition}
\newtheorem{defn}[thm]{Definition}
\theoremstyle{remark}
\newtheorem{rem}[thm]{Remark}
\numberwithin{equation}{section}
\newcounter{stepnum}

\def\bee{\begin{eqnarray}}
\def\beee{\begin{eqnarray*}}
\def\eee{\end{eqnarray}}
\def\eeee{\end{eqnarray*}}

\def\ba{\begin{array}}
\def\ea{\end{array}}

\def\R{\mathbb R}

\begin{document}

\title[Chern-Simon-Higgs]{The qualitative behavior at a vortex point for the Chern-Simon-Higgs equation}

\author[Li]{Jiayu Li}%
\address{School of Mathematical Sciences, University of Science and Technology of China, Hefei, 230026, P. R. China}%
\email{jiayuli@ustc.edu.cn}%

\author[Liu]{Lei Liu}
\address{School of Mathematics and Statistics \& Hubei Key Laboratory of
Mathematical Sciences, Central China Normal University, Wuhan 430079,
P. R. China}%
\email{leiliu2020@ccnu.edu.cn}

\thanks{The work was carried out when Lei Liu was visiting the China-France Mathematics Center,  University of Science and Technology of China. He would like to thank the center for the hospitality and the good working conditions.}

%\thanks{}
\subjclass[2010]{}
\keywords{Chern-Simon-Higgs equation, Blow up, Energy identity, Vortex point}

\date{\today}
%\dedicatory{}%
%\commby{}%
% ----------------------------------------------------------------
\begin{abstract}
In this paper, we study the qualitative behavior at a vortex blow-up point for Chern-Simon-Higgs equation. Roughly speaking, we will establish an energy identity at a each such point, i.e. the local mass is the sum of the bubbles. Moreover, we prove that either there is only one bubble which is a singular bubble or there are more than two bubbles which contains no singular bubble. Meanwhile, we prove that the energies of these bubbles must satisfy a quadratic polynomial which can be used to prove the simple blow-up property when the multiplicity is small. As is well known, for many Liouville type system, Pohozaev type identity is a quadratic polynomial corresponding to energies which can be used directly to compute the local mass at a blow-up point. The difficulty here is that, besides the energy's integration,  there is a additional term in the Pohozaev type identity of Chern-Simon-Higgs equation. We need some more detailed and delicated analysis to deal with it.
\end{abstract}
\maketitle
% ----------------------------------------------------------------
\section{introduction}

\

The present model deals with a simplified form of the so-called anyon model, a classical field theory defined on $(2+1)$ dimensional Minkowski apace where the Lagrangian couples Maxwell and Chern-Simons terms coming from a gauge field with a scalar field. Generally, the Euler-Lagrange equations of this model are complicated. For a simplified version without the Maxwell term, Hong-Kim-Pac \cite{Hong-Kim-Pac} and Jackiw-Weinberg \cite{Jackiw-Weinberg} proposed a special choice of $6^{th}$ order potential where there may exist time-independent vortex solutions satisfying a first order Bogomolny type self-dual equations.

Mathematically, Chern-Simons-Higgs 6-th model can be described as follows. Let $M$ be a compact Riemann surface with volume $1$ and $L$ be a line bubble over $M$. For a unitary connection $D_A=d+A$ on $L$ with curvature $F_A$, and a section $\phi$ of $L$,  Hong-Kim-Pac \cite{Hong-Kim-Pac} and Jackiw-Weinberg \cite{Jackiw-Weinberg} introduced Chern-Simons-Higgs 6-th functional as follows:
\begin{align*}
CS(A,\phi)&=\int_M\left(|D_A\phi|^2+\frac{\epsilon^2}{|\phi|^2}|F_A|^2+\frac{1}{4\epsilon^2}|\phi|^2(1-|\phi|^2)^2\right)dV
\\&=\int_M\left(|(D_1+iD_2)\phi|^2+\left(\frac{\epsilon}{|\phi|}F-\frac{1}{2\epsilon}|\phi|(1-|\phi|^2)\right)^2\right)dV+2\pi\ deg\ L.
\end{align*}

The absolute minimizers of above functional satisfy the Bogomolny type self-dual equations
\begin{align}
\begin{cases}
(D_1+iD_2)\phi&=0,\\
\frac{1}{2}|\phi|^2(1-|\phi|^2)&=\epsilon^2 F.
\end{cases}
\end{align}
Putting $$v:=\log |\phi|^2,$$ we reduce the above system to the following single scalar equation that
\begin{equation}\label{equ:chern-simons}
\Delta v=\frac{4}{\epsilon^2}e^v(e^v-1)+4\pi\sum_{j=1}^JN_j\delta_{p_j},\ \ in\ \ M,
\end{equation}
where $p_j\in M$, $N_j$ is a positive integer, $\delta_{p_j}$ is the Dirac measure at $p_j$, $j=1,...,J$. For more details of deriving this equation, we refer to \cite{Hong-Kim-Pac,Jackiw-Weinberg,Ding-Jost-Li-Wang-1,Nolasco-Tarantello,Tarantello}.

In this paper, we want to study the compactness problem for a sequence of solutions of \eqref{equ:chern-simons}. Let $v_n$ satisfies
\begin{equation}
\Delta v_n(x)=\lambda_ne^{v_n(x)}(e^{v_n(x)}-1)+4\pi\sum_{j=1}^JN_j\delta_{p_j},
\end{equation} where $N_1,...,N_J$ are positive integers, $\sum_{j=1}^JN_j=\bar{N}$ and $\lim_{n\to\infty}\lambda_n\to +\infty$.

Let $u_0$ be the corresponding Green function, i.e.
\begin{align*}
\begin{cases}
\Delta u_0&=-4\pi \bar{N}+4\pi\sum_{j=1}^JN_j\delta_{p_j},\\
\int_Mu_0dM&=0.
\end{cases}
\end{align*}
Then $u_n=v_n-u_0$ satisfies
\begin{equation}\label{equ:chern-simon-2}
-\Delta u_n=\lambda_ne^{u_0} e^{u_n}(1-e^{u_0} e^{u_n})-4\pi \bar{N}.\end{equation}

Denote $$\bar{u}_n=\int_Mu_ndM,\ \ w_n(x):=u_n(x)-\bar{u}_n.$$ In \cite{Lan-Li,Choe-Kim}, they proved the following theorem that
\begin{thm}[\cite{Lan-Li,Choe-Kim}]\label{thm:1}
Let $M$ be a compact Riemann surface with volume $1$.  Then passing to a subsequence, the following alternatives hold.
\begin{itemize}
\item[(1)] If $\lim_{n\to\infty}\lambda_ne^{\bar{u}_n}=+\infty$, we have $u_n\to -u_0$ in the sense of $W^{1,p}(M),\ (1<p<2)$ as $n\to\infty$.

    \

\item[(2)] If $\lim_{n\to\infty}\lambda_ne^{\bar{u}_n}=c>0$, we have $w_n\to w_0$ weakly in  $W^{1,p}(M),\ (1<p<2)$ as $n\to\infty$, where $w_0$ satisfies $$-\Delta w_0=ce^{u_0}e^{w_0}-4\pi\bar{N}+4\pi\sum_{i=1}^{I}\alpha_i\delta_{p_{k_i}},$$ with $\int_M w_0dM=0$, $1\leq \alpha_i\leq N_{k_i}$, $N_{k_i}$ is the multiplicity of the vortex $p_{k_i}$.

    \

\item[(3)] If  $\lim_{n\to\infty}\lambda_ne^{\bar{u}_n}=c=0$, we have $w_n\to \tilde{w}$ in  $W^{1,p}(M),\ (1<p<2)$ as $n\to\infty$, where  $\tilde{w}$ is a positive Green function satisfying $$-\Delta
    \tilde{w}=-4\pi\bar{N}+4\pi\sum_{i=1}^{I}\beta_i\delta_{q_i},$$ with $\int_M \tilde{w}dM=0$, $1\leq \beta_i\leq \bar{N}$ and $\sum_{i=1}^I\beta_i=\bar{N}$.
\end{itemize}

\end{thm}

\

Recently, Lin-Yan \cite{Lin-Yan,Lin-Yan-1} studied doubly periodic solutions $u_n$ for \eqref{equ:chern-simon-2} in $\Omega$ with the boundary condition
\begin{equation}\label{equ:21}
u_n \mbox{  is doubly periodic on }\partial \Omega.
\end{equation}
Similar to classical theory of Liouville equations, they \cite{Lin-Yan,Lin-Yan-1} proved that if the blow-up point is not a vortex point, then there is only one bubble, i.e. it must be simple. Furthermore,  some interesting bubbling solutions were also constructed in their papers.

\

In this paper, we continue to study the blow-up picture of \eqref{equ:chern-simon-2} by focusing on following two issues:
\begin{itemize}
\item[(1)] If the blow-up point is not a vortex point, the boundary condition \eqref{equ:21} can be changed to a general one, e.g. $osc_{\partial \Omega}u_n\leq C$, which is more nature from the blow-up theory \cite{Brezis,Choe-Kim,Lan-Li}?

\

\item[(2)] If the blow-up point is a vortex point, can we get a energy identity, i.e. the local mass is the sum of  bubbles? Furthermore, if the simple blow-up property still hold at a vortex blow-up point?
\end{itemize}

\

Integrating \eqref{equ:chern-simon-2}, we get
\begin{equation}\label{ineq:14}
\int_M \lambda_ne^{u_0} e^{u_n}(1-e^{u_0} e^{u_n})dM=4\pi \bar{N}.\end{equation}
By a standard maximal argument in \cite{Caffarelli-Yang}, it is easy to see  that $v_n=u_0+u_n\leq 0$. Furthermore, if $\lim\lambda_ne^{\bar{u}_n}=c<\infty$, by Proposition 4.1 of \cite{Choe-Kim}, we get that $$v_n(x)\leq -C$$ for some positive constant $C>0$, which immediately implies that
\begin{equation}\label{ineq:15}
\int_M \lambda_ne^{u_0} e^{u_n}dM\leq \frac{4\pi \bar{N}}{1-e^{-C}}.
\end{equation}

Under the assumptions of Theorem \ref{thm:1} and \eqref{ineq:14}-\eqref{ineq:15}, at each blow-up point $q$ (i.e. there exists $x_n\to q$ such that $w_n(x_n)\to +\infty$), after a suitable rescaling argument, Lan-Li \cite{Lan-Li} proved that there is a bubble $\varphi(x)$ splitting off, which satisfies the equation
\begin{align}\label{equ:07}
\begin{cases}
-\Delta\varphi(x)=e^{\varphi(x)}\left(1-\delta e^{\varphi(x)}\right),\ \ x\in\R^2,\\
\delta e^{\varphi(x)}\leq 1, \ \int_{\R^2}e^{\varphi(x)}\left(1-\delta e^{\varphi(x)}\right)dx +\int_{\R^2}e^{\varphi(x)}dx\leq C,
\end{cases}
\end{align} where $\delta\geq 0$ is a constant, or satisfy the equation
\begin{align}\label{equ:04}
\begin{cases}
-\Delta\varphi(x)=|x-x_0|^{2N_j}e^{\varphi(x)}\left(1-\delta |x-x_0|^{2N_j}e^{\varphi(x)}\right),\ \ x\in\R^2,\\
\delta |x-x_0|^{2N_j}e^{\varphi(x)}\leq 1,\\
 \int_{\R^2}|x-x_0|^{2N_j}e^{\varphi(x)}\left(1-\delta |x-x_0|^{2N_j}e^{\varphi(x)}\right)dx +\int_{\R^2}|x-x_0|^{2N_j}e^{\varphi(x)}dx\leq C;
 \end{cases}
\end{align}where $\delta\geq 0$ is a constant (maybe different from the previous one), $x_0\in\R^2$.

\

In this paper, near a blow-up point $q$, if there exist $x_n\to q$ and $\mu_n\to 0$ such that
\begin{equation}\label{equ:17}
\varphi_n(x):=w_n(x_n+\mu_nx)-w_n(x_n)\to \varphi(x),\ \ in\ \ C^2_{loc}(\R^2),\end{equation} where $\varphi(x)$ satisfies \eqref{equ:04} or \eqref{equ:07}, then we say $\varphi(x)$ or $(x_n,\mu_n)$ is a bubble. Specially, we say $\varphi(x)$ or $(x_n,\mu_n)$  a singular bubble if $\varphi(x)$ satisfies \eqref{equ:04}.

\

Denote $$\alpha_n(s_n;x_n):= \frac{1}{2\pi}\int_{B_{s_n}(x_n)}\lambda_ne^{\bar{u}_n}e^{u_0(x)}e^{w_n}[1-e^{\bar{u}_n}e^{u_0(x)}e^{w_n}]dx.$$

Our first main result is

\begin{thm}\label{thm:main}
Under the assumptions and notations of Theorem \ref{thm:1}, assuming $c=\lim_{n\to\infty}\lambda_ne^{\bar{u}_n}<+\infty$, we have:
\begin{itemize}
\item[(1)] if $c>0$, then $I=0$, that is $w_n\to w_0$ converges strongly in $C^2(M)$.

\

\item[(2)] if $c=0$ and  $w_n$ (or $u_n$) blows up at a  point $q_i$ which is not a vortex point, then there is only one bubble $\varphi(x)$ satisfying \eqref{equ:07}, i.e. the simple blow-up property holds. Moreover, we have following energy identity that $$\lim_{\rho\to 0}\lim_{n\to\infty}\alpha_n(\rho; q_i)=b,$$ where $$b=\frac{1}{2\pi}\int_{\R^2}e^{\varphi(x)}\left(1-\delta e^{\varphi(x)}\right)dx\geq 4.$$

\

\item[(3)] if $c=0$ and  $w_n$ (or $u_n$) blows up at a vortex point $p_i$,  then the following two alternatives hold.
\begin{itemize}
\item[(3-1)] If there is a singular bubble $\varphi(x)$ near the blow-up point $p_i$, then $\varphi(x)$ is the only bubble, i.e. the simple blow-up property holds. Moreover, we have following energy identity $$\lim_{\rho\to 0}\lim_{n\to\infty}\alpha_n(\rho ;p_i)=b,$$ where $$b=\frac{1}{2\pi}\int_{\R^2}|x-x_0|^{2N_i}e^{\varphi(x)}\left(1-\delta |x-x_0|^{2N_i}e^{\varphi(x)}\right)dx \geq 4(1+N_i)$$ and $N_i$ is the multiplicity of vortex point $p_i$.

    \

\item[(3-2)] If there is a non-singular bubble $\varphi^1(x)$ near the blow-up point $p_i$, then there must exist more than two bubbles $\{\varphi^k\}_{k=1}^L$. Moreover $\{\varphi^k\}_{k=1}^L$ contains no singular bubble and satisfies the following energy identity $$\lim_{\rho\to 0}\lim_{n\to\infty}\alpha_n(\rho; p_i)=\sum_{k=1}^Lb_k,$$ where $$b_k=\frac{1}{2\pi}\int_{\R^2}e^{\varphi^k(x)}\left(1-\delta e^{\varphi^k(x)}\right)dx\geq 4.$$
 Furthermore, energies of these bubbles satisfy the following relation that
 \begin{equation}\label{equ:09}
 \sum_{k=1}^Lb_k^2=\left(\sum_{k=1}^Lb_k\right)^2-4N_i\sum_{k=1}^Lb_k.
 \end{equation}
\end{itemize}
\end{itemize}
\end{thm}

\

To prove Theorem \ref{thm:main}, we first need to extend some lemmas in the blow-up theory of Liouville equations to Chern-Simon-Higgs equation, including a selection of bubbling areas, fast decay lemma and two kinds of Pohozaev type identities. This kind of selection of bubbling areas was firstly used in Li-Shafrir \cite{Li-Shafrir} for Liouville equations where they showed at the centers $\Sigma_n$ of bubbling disks, the singularity is $-2log \ dist(x,\Sigma_n)$, i.e. $w_n(x)+2\log  \ dist(x,\Sigma_n)$ is uniformly bounded from above. However, this process can not be extended easily to Chern-Simon-Higgs equation, since the latter is more complicated due to the different singularities at centers of bubbling disks. We need to introduce a new test function and some classification discussion, see Proposition \ref{prop:bubble-tree}. To prove fast decay lemma, we need to estimate the oscillation of $$w_n(x)+2\log dist (x,\Sigma_n)+2N\log |x|$$ in a annulus containing zero point, where one can see that the term $\log |x|$ will take trouble to us. For Pohozaev type identity, it is well known that for Liouville equation, this identity is a quadratic polynomial with respect to the energy which implies immediately the value of local mass. For Chern-Simon-Higgs equation, on one hand, there are two kinds of  identities, local type together with global type. On the other hand, besides the energy term, there is a additional integration $$\int_{B_{s_n}(x_n)}\lambda_ne^{\bar{u}_n}e^{u_0(x)}e^{w_n}[1-\frac{1}{2}e^{\bar{u}_n}e^{u_0(x)}e^{w_n}]dx$$  in both two kinds of Pohozaev type identities, see Lemma \ref{lem:pohozaev}.  This is a main difficulty in the blow-up analysis of Chern-Simon-Higgs equation. To overcome this obstacle, we will use some ideas in the blow-up theory of harmonic maps \cite{Ding-Tian} by proving that there is no energy loss in neck domains.

Secondly, since the main proof is a induction argument on the  bubbling disks, we need to explore the positional structures of these bubbling areas. This will be solved by using a argument of local and global type Pohozaev identities, where the relation \eqref{equ:09} follows, see Section \ref{sec:lemmas}. We want to remark that in case $(3-1)$ of Theorem \ref{thm:main}, by using this type argument, we can prove the simple blow-up phenomenon. But generally, we can not exclude case $(3-2)$ of Theorem \ref{thm:main}.

\

In the blow-up theory of Liouville type system, there is an important question that if the simple blow-up property holds near a blow-up point, i.e. there is only one bubble at each blow-up point?  This kind result was firstly prove by Li \cite{Li} for Liouville equation. Later, Jost-Lin-Wang \cite{Jost-Lin-Wang} proved it also holds for Toda system. For Liouville equation with quantized singular sources, see  \cite{Bartolucci-Tarantello-1,Bartolucci-Tarantello-2,Bartolucci-Chen-Lin-Tarantello,D'Aprile-Wei,D'Aprile-Wei-Zhang,Wei-Zhang-1,Wei-Zhang-2}.

For Chern-Simon-Higgs equation, by Lin-Yan \cite{Lin-Yan} and above Theorem \ref{thm:main}, we know  the simple blow-up property holds near a non-vortex blow-up point. At a vortex blow-up point, Theorem \ref{thm:main} tells us that either $(3-1)$ or $(3-2)$ happens and generally we can not exclude the second case. But, when the sum of multiplicities $\bar{N}=\sum_{j=1}^JN_j$ is small, we can use identity \eqref{equ:09} in Theorem \ref{thm:main}  to show that there is one bubble near a vortex blow-up point which must be a singular bubble. The following theorem may be useful when one want to construct non-simple bubbling solutions.

\

\begin{thm}\label{thm:main-3}
Denote the blow-up set by $S:=\{q_1,...,q_I\}$. Assume $S$ is not empty. Then we have
\begin{itemize}
\item[(1)] When $\bar{N}\leq 3$, $S$ contains no vortex point.

\

\item[(2)] When $\bar{N}=5$, if $q_1$ is a vortex point, then $S=\{q_1\}$ and its multiplicity $N_1=1$. In this case, the simple blow-up property holds.

    \

\item[(3)] When $\bar{N}=6$, if $S=\{q_1\}$ where $q_1$ is a vortex point with multiplicity $N_1=1$, then the simple blow-up property holds.

    \

\item[(4)] When $\bar{N}=7$, if $S=\{q_1\}$ where $q_1$ is a vortex point, then $N_1=1$ or $N_1=2$ and the simple blow-up property holds.
\end{itemize}
\end{thm}

\

At the end of this part, we want to ask the following:

\

\textbf{Question:} For Chern-Simon-Higgs equation, if the simple blow-up property holds at a vortex blow-up point?

\

The rest of the paper is organized as follows. In Section \ref{sec:lemmas}, we will prove some basic lemmas which will be used in our proof, such as a selection of bubbling areas, oscillation lemma, fast decay lemma and two kinds of Pohozaev type identities. The proof of Theorem \ref{thm:main} and Theorem \ref{thm:main-3} will be given in Section \ref{sec:proof}.

\

\section{Some basic lemmas}\label{sec:lemmas}

\

In this section, we will first divide  Theorem \ref{thm:main} into two local versions, i.e. Theorem \ref{thm:main-1} and Theorem \ref{thm:main-2}. Secondly, we will establish some lemmas which will be used in our proof, including a selection of bubbling areas, oscillation lemma, fast decay lemma and two kinds of Pohozaev type identities. Lastly, we will prove a local energy identity.

\

For simplicity of notations, we assume $0$ is a vortex point with multiplicity $N$. Writing Green function $$u_0(x)=2N\log |x|+\zeta(x)$$ where $\zeta(x)$ is a smooth function near point $0$, then we consider the following system
\begin{align}\label{equ:01}
  -\Delta u_n&=\lambda_n|x|^{2N}e^{\zeta(x)} e^{u_n}(1-|x|^{2N} e^{\zeta(x)} e^{u_n}),\ \ in\ \ B_1(0),
  \end{align}
  where $\lambda_n$ is a sequence of numbers which converges to $+\infty$, and $$ \int_{B_1(0)}\lambda_n|x|^{2N} e^{\zeta(x)}e^{u_n}(1-|x|^{2N} e^{\zeta(x)} e^{u_n})dx+ \int_{B_1(0)}\lambda_n|x|^{2N}  e^{\zeta(x)}e^{u_n}dx\leq \Lambda$$ and $$|x|^{2N} e^{\zeta(x)} e^{u_n}\leq 1,\  \ osc_{\partial B}u_n\leq C.$$

Denote $$\bar{u}_n:=\frac{1}{\pi}\int_{B_1(0)}u_n(x)dx,\ \ w_n(x):=u_n(x)-\bar{u}_n,$$ then we have
\begin{align}\label{equ:02}
  -\Delta w_n&=\lambda_n e^{\bar{u}_n}|x|^{2N} e^{\zeta(x)} e^{w_n}(1-e^{\bar{u}_n} |x|^{2N}e^{\zeta(x)} e^{w_n}),\ \ in\ \ B_1(0),
  \end{align} and
  \begin{equation}\label{equ:06}
  \int_{B_1(0)}\lambda_n|x|^{2N} e^{\zeta(x)}e^{\bar{u}_n+w_n(x)}(1-|x|^{2N} e^{\zeta(x)} e^{\bar{u}_n+w_n(x)})dx+ \int_{B_1(0)}\lambda_n|x|^{2N}  e^{\zeta(x)}e^{\bar{u}_n+w_n(x)}dx\leq \Lambda
   \end{equation} and
   \begin{equation}\label{equ:08}
   |x|^{2N} e^{\zeta(x)} e^{\bar{u}_n+w_n(x)}\leq 1,\  \ osc_{\partial B}w_n\leq C.
   \end{equation}

Assume $0$ is the only blow-up point in $B_1(0)$ for sequence $w_n(x)$, i.e.
\begin{equation}\label{equ:15}
\max_{K\subset\subset \overline{B_1(0)}\setminus \{0\}}|w_n(x)|\leq C,\ \ \max_{ \overline{B_1(0)}}|w_n(x)|\to+\infty.
\end{equation}

\

We will prove the following two theorems of a local version of Theorem \ref{thm:main}.

\

\begin{thm}\label{thm:main-1}
Let $w_n$ be a sequence of solutions of \eqref{equ:02} satisfying \eqref{equ:06}, \eqref{equ:08} and \eqref{equ:15}. Assume $$c=\lim_{n\to\infty}\lambda_ne^{\bar{u}_n}<\infty.$$ If $0$ is a vortex point, i.e. $N>0$, then we have $c=0$ and the following two alternatives hold:
\begin{itemize}
\item[(1)]  if there is a singular bubble $\varphi(x)$ near the blow-up point $0$, then $\varphi(x)$ is the only bubble, i.e. the simple blow-up property holds.  Moreover, we have the following energy identity $$\lim_{\rho\to 0}\lim_{n\to\infty}\alpha_n(\rho;0)=b,$$ where $$b=\frac{1}{2\pi}\int_{\R^2}|x-x_0|^{2N}e^{\varphi(x)}\left(1-\delta |x-x_0|^{2N}e^{\varphi(x)}\right)dx \geq 4(1+N).$$

    \

\item[(2)] if there is a non-singular bubble $\varphi^1(x)$ near the blow-up point $0$, then there must exist more than two bubbles $\{\varphi^i\}_{i=1}^m$ near the blow-up point $0$. Moreover $\{\varphi^i\}_{i=1}^m$ contains no singular bubble and satisfies the following energy identity $$\lim_{\rho\to 0}\lim_{n\to\infty}\alpha_n(\rho;0)=\sum_{i=1}^mb_i,$$ where $$b_i=\frac{1}{2\pi}\int_{\R^2}e^{\varphi^i(x)}\left(1-\delta_i e^{\varphi^i(x)}\right)dx\geq 4,\ \ i=1,...,m.$$ Furthermore, energies $b_1,...,b_m$ satisfy the following relation that $$\sum_{i=1}^mb_i^2=\left(\sum_{i=1}^mb_i\right)^2-4N\sum_{i=1}^mb_i.$$
\end{itemize}

\end{thm}

\

When $0$ is not a vortex point, i.e. $N=0$, we have
\begin{thm}\label{thm:main-2}
Let $w_n$ be a sequence of solutions of \eqref{equ:02} satisfying \eqref{equ:06}, \eqref{equ:08} and \eqref{equ:15}. If $0$ is not a vortex point, i.e.  $N=0$, then $c=0$ and  there is only one bubble $\varphi(x)$ near the blow-up point $0$, satisfying \eqref{equ:07} and following energy identity that $$\lim_{\rho\to 0}\lim_{n\to\infty}\alpha_n(\rho;0)=b,$$ where $$b=\frac{1}{2\pi}\int_{\R^2}e^{\varphi(x)}\left(1-\delta e^{\varphi(x)}\right)dx\geq 4.$$
\end{thm}

\

Next, we establish some lemmas which will be used in our later proof. We start by proving the following lemma for a lower bound of energy concentration.

\begin{lem}\label{lem:02}
Let $w_n$ be a sequence of solutions of \eqref{equ:02} satisfying \eqref{equ:06}, \eqref{equ:08} and \eqref{equ:15}. Then we have  $c=0$ and $$\lim_{\rho\to 0}\lim_{n\to\infty}\alpha_n(\rho;0)\geq 4(1+N).$$
\end{lem}
\begin{proof}
By Theorem \ref{thm:1}, we know $w_n\to w_0$ weakly in $W^{1,p}(B_1(0))$ and strongly in $C^2_{loc}(B_1(0)\setminus \{0\})$, where $w_0$ satisfies $$-\Delta w_0=ce^{u_0}e^{w_0}+4\pi\beta\delta_0,\ \ in\ \ B_1(0)$$ and $\beta\geq 1$. Moreover, $$\lambda_n e^{\bar{u}_n}|x|^{2N} e^{\zeta(x)} e^{w_n}(1-e^{\bar{u}_n} |x|^{2N}e^{\zeta(x)} e^{w_n})\to ce^{u_0}e^{w_0}+4\pi\beta\delta_0$$ in the sense of measure which implies that $$\alpha_n(\rho;0)=2\beta+o(1),$$ where $\lim_{\rho\to 0}\lim_{n\to\infty}o(1)=0$.

For any $\rho\in (0,\frac{1}{2})$, multiplying \eqref{equ:02} by $x\nabla w_n$ and integrating by parts, we get
\begin{align}\label{equ:22}
\int_{\partial B_{\rho}(0)}r\left(\frac{1}{2}|\nabla w_n|^2-|\frac{\partial w_n}{\partial r}|^2\right)d\theta=& \int_{\partial B_{\rho}(0)}r\lambda_ne^{\bar{u}_n}|x|^{2N}e^{\zeta}e^{w_n}[1-\frac{1}{2}e^{\bar{u}_n}|x|^{2N}e^{\zeta}e^{w_n}]d\theta\notag\\ &- \int_{B_{\rho}(0)}2\lambda_ne^{\bar{u}_n}|x|^{2N}e^{\zeta}e^{w_n}[1-\frac{1}{2}e^{\bar{u}_n}|x|^{2N}e^{\zeta}e^{w_n}]dx\notag\\
& - \int_{B_{\rho}(0)}2N\lambda_ne^{\bar{u}_n}|x|^{2N}e^{\zeta}e^{w_n}[1-e^{\bar{u}_n}|x|^{2N}e^{\zeta}e^{w_n}]dx\notag\\
& - \int_{B_{\rho}(0)}\lambda_ne^{\bar{u}_n}|x|^{2N}( x\cdot \nabla \zeta )e^{\zeta}e^{w_n}[1-e^{\bar{u}_n}|x|^{2N}e^{\zeta}e^{w_n}]dx.
\end{align}
Since we can write $$w_0(x)=-2\beta\log |x|+\phi(x)$$ where $\phi(x)$ is a smooth function, then $$\int_{\partial B_{\rho}(0)}r\left(\frac{1}{2}|\nabla w_n|^2-|\frac{\partial w_n}{\partial r}|^2\right)d\theta=-4\pi\beta^2+o(1)=-\pi\alpha_n^2(\rho;0)+o(1).$$

Noting that $| x\cdot \nabla \zeta|\leq C\rho$ yields $$\lim_{\rho\to 0}\lim_{n\to\infty}\int_{B_{\rho}(0)}\lambda_ne^{\bar{u}_n}|x|^{2N}( x\cdot \nabla \zeta )e^{\zeta}e^{w_n}[1-e^{\bar{u}_n}|x|^{2N}e^{\zeta}e^{w_n}]dx=0,$$ we have
\begin{align*}
-\pi\alpha_n^2(0;\rho) &\leq - \int_{B_{\rho}(0)}2\lambda_ne^{\bar{u}_n}|x|^{2N}e^{\zeta}e^{w_n}[1-\frac{1}{2}e^{\bar{u}_n}|x|^{2N}e^{\zeta}e^{w_n}]dx\\
& \quad - \int_{B_{\rho}(0)}2N\lambda_ne^{\bar{u}_n}|x|^{2N}e^{\zeta}e^{w_n}[1-e^{\bar{u}_n}|x|^{2N}e^{\zeta}e^{w_n}]dx+o(1)\\
& \leq - 2(1+N)\int_{B_{\rho}(0)}\lambda_ne^{\bar{u}_n}|x|^{2N}e^{\zeta}e^{w_n}[1-e^{\bar{u}_n}|x|^{2N}e^{\zeta}e^{w_n}]dx =-4\pi (1+N)\alpha_n(\rho;0)+o(1)
\end{align*} which implies that $$\lim_{\rho\to 0}\lim_{n\to\infty}\alpha_n(\rho;0)\geq 4(1+N).$$

Now, we claim that $c=0$. In fact, if not, by Theorem \ref{thm:1}, we have $\beta\leq N$. Then there holds $$4(1+N)\leq \lim_{\rho\to 0}\lim_{n\to\infty}\alpha_n(\rho;0)= 2\beta\leq 2N.$$ This is a contradiction. Thus, $c=0$.
\end{proof}

\

Next, we prove a proposition for a selection of bubbling areas. Different from Liouville equations \cite{Li-Shafrir} where the singularity near the center $p$ of each bubbling disk is $-2\log|x-p|$, Chern-Simon-Higgs equations will involve two kinds of singularities which is important in our later proof.

\begin{prop}\label{prop:bubble-tree}
Let $w_n$ be a sequence of solutions of \eqref{equ:02} satisfying \eqref{equ:06}, \eqref{equ:08} and \eqref{equ:15}. Then there exist a sequence of finite points $\Sigma_n:=\{x^1_n,x^2_n,...,x^m_n\}$ and a sequence of positive numbers $\beta^1_n,\beta^2_n,...,\beta^m_n$ such that
\begin{itemize}
\item[(1)] $x^i_n\to 0$ and $\beta_n^i\to 0$ as $n\to\infty$, $\beta^i_n\leq\frac{1}{2}dist (x^i_n,\Sigma_n\setminus \{x^i_n\})$, $i=1,...,m$;

\

\item[(2)] $w_n(x^i_n)=\max_{B_{\beta_n^i}(x^i_n)}w_n(x)\to +\infty$ as $n\to\infty$, $i=1,...,m$;

\

\item[(3)] Take
\begin{align*}
\mu_n^i:=
\begin{cases}
\left(\lambda_ne^{\bar{u}_n+w_n(x^i_n)}e^{\zeta(x_n^i)}\right)^{-\frac{1}{2(1+N)}},\ \ &if\ \ |x^i_n|\left(\lambda_ne^{\bar{u}_n+w_n(x^i_n)}\right)^{\frac{1}{2(1+N)}}\leq C;\\
\left(\lambda_ne^{\bar{u}_n+w_n(x^i_n)}|x^i_n|^{2N}e^{\zeta(x_n^i)}\right)^{-\frac{1}{2}},\ \ &if\ \ |x^i_n|\left(\lambda_ne^{\bar{u}_n+w_n(x^i_n)}\right)^{\frac{1}{2(1+N)}}\to \infty.
\end{cases}
\end{align*} Then $\frac{\mu_n^i}{\beta^i_n}\to 0$ as $n\to\infty$, $i=1,...,m$.

\

\item[(4)] In each $B_{\beta_n^i}(x^i_n)$, we define the scaled function  $$\varphi^i_n(x):=w_n(x^i_n+\mu_n^ix)-w_n(x_n^i).$$ Then we have the following two cases:

\begin{itemize}
\item[(4-1)] if $|x^i_n|\left(\lambda_ne^{\bar{u}_n+w_n(x^i_n)}\right)^{\frac{1}{2(1+N)}}\leq C$, then $\varphi^i_n(x)\to \varphi^i(x)$ in $C^2_{loc}(\R^2)$ as $n\to\infty$, where $\varphi^i(x)$ satisfies \eqref{equ:04}.

\item[(4-2)] if $|x^i_n|\left(\lambda_ne^{\bar{u}_n+w_n(x^i_n)}\right)^{\frac{1}{2(1+N)}}\to \infty$, then $\varphi^i_n(x)\to \varphi^i(x)$ in $C^2_{loc}(\R^2)$ as $n\to\infty$, where $\varphi^i(x)$ satisfies \eqref{equ:07}.

\end{itemize}

\item[(5)] There is at most one point $x^i_n\in\Sigma_n$ such that $$|x^i_n|\left(\lambda_ne^{\bar{u}_n+w_n(x^i_n)}\right)^{\frac{1}{2(1+N)}}\leq C.$$     Furthermore, we have the following:
    \begin{itemize}
    \item[(5-1)] if there is one point $x^i_n\in\Sigma_n$ such that $|x^i_n|\left(\lambda_ne^{\bar{u}_n+w_n(x^i_n)}\right)^{\frac{1}{2(1+N)}}\leq C,$ then we have $$0\in B_{\beta_n^i}(x_n^i)\ \ and \ \ \lim\frac{|x_n^j|}{|x_n^i|}=\infty,\ \forall j\neq i.$$ Moreover,  for any positive constant $c_0>0$, there holds  \begin{align*}w_n(x)+2\log dist(x,\Sigma_n)+2N\log |x|+\log\lambda_ne^{\bar{u}_n}\leq C(c_0),\ \ \forall |x-x_n^i|\geq c_0\mu_n^i.\end{align*}
    \item[(5-2)] if for any $x^i_n\in\Sigma_n$, there holds $|x^i_n|\left(\lambda_ne^{\bar{u}_n+w_n(x^i_n)}\right)^{\frac{1}{2(1+N)}}\to\infty,$ then \begin{align*}w_n(x)+2\log dist(x,\Sigma_n)+2N\log |x|+\log\lambda_ne^{\bar{u}_n}\leq C,\  \ \forall \ x\in B_1(0).\end{align*}
    \end{itemize}

\end{itemize}
\end{prop}

\begin{proof}
\textbf{Step 1:} Construction of the first blow-up point $x_n^1$ and $\mu_n^1$.

\

Choose $x_n\in B_{\frac{1}{2}}(0)$ such that  $x_n\to 0$ and $$w_n(x_n)=\max_{\bar{B}}w_n(x)\to +\infty.$$ Denote $\gamma_n:=w_n(x_n)$ and $$\varphi_n(x)=w_n(x_n+\mu_nx)-\gamma_n,$$ where $\mu_n$ is a positive number going to zero which will be choosed  later.

A direct computation yields
\begin{align*}
&-\Delta\varphi_n(x)\\&=\mu_n^2\lambda_n e^{\bar{u}_n+\gamma_n}|x_n+\mu_nx|^{2N}e^{\zeta(x_n+\mu_nx)} e^{\varphi_n(x)}
(1-e^{\bar{u}_n+\gamma_n}|x_n+\mu_nx|^{2N} e^{\zeta(x_n+\mu_nx)}e^{\varphi_n(x)})\\
&=\mu_n^{2(1+N)}\lambda_n e^{\bar{u}_n+\gamma_n}|\frac{x_n}{\mu_n}+x|^{2N} e^{\zeta(x_n+\mu_nx)}e^{\varphi_n(x)}
\left(1-\mu_n^{2N}e^{\bar{u}_n+\gamma_n}|\frac{x_n}{\mu_n}+x|^{2N} e^{\zeta(x_n+\mu_nx)}e^{\varphi_n(x)}\right),\ \ in\ \ B_{\frac{1}{2}\mu_n^{-1}}(0).
\end{align*}

\

\textbf{Case 1:} $\lambda_n e^{\bar{u}_n+\gamma_n}|x_n|^{2(1+N)}\leq C$.

\

In this case, we first claim: $$\lambda_n^{-N} e^{\bar{u}_n+\gamma_n}\leq C.$$

\

In fact, if not, passing to a subsequence, there holds $\lim_{n\to\infty}\lambda_n^{-N} e^{\bar{u}_n+\gamma_n}=\infty$. Taking $$\mu_n=\left(\lambda_ne^{2(\bar{u}_n+\gamma_n)}e^{2\zeta(x_n)}\right)^{-\frac{1}{4N+2}},$$ then we have \begin{align*}
\frac{|x_n|}{\mu_n}&=\left(|x_n|^{4N+2}\lambda_ne^{2(\bar{u}_n+\gamma_n)}e^{2\zeta(x_n)}\right)^{\frac{1}{4N+2}}\\
&=\left(|x_n|^{2N+2}\lambda_ne^{\bar{u}_n+\gamma_n} |x_n|^{2N}e^{\bar{u}_n+\gamma_n}e^{2\zeta(x_n)}\right)^{\frac{1}{4N+2}}\leq C \left( |x_n|^{2N}e^{\bar{u}_n+\gamma_n}e^{2\zeta(x_n)}\right)^{\frac{1}{4N+2}}\leq C
\end{align*} where the last inequality follows from the fact that $v_n(x_n)=\bar{u}_n+\gamma_n+u_0(x_n)\leq 0$,  and $$\mu_n^{2(1+N)}\lambda_ne^{\bar{u}_n+\gamma_n}=e^{-\frac{2(N+1)}{2N+1}\zeta(x_n)}\left(\lambda_n^Ne^{-\bar{u}_n-\gamma_n} \right)^{\frac{1}{2N+1}}\to 0,$$ which implies $\mu_n\to 0$, because $\lambda^{-N}_ne^{\bar{u}_n+\gamma_n}\to \infty$.

Passing to a subsequence, we assume $\lim_{n\to\infty}\frac{x_n}{\mu_n}=-x_0$. Since $\varphi_n(x)\leq 0,\ \varphi_n(0)=0$, by standard elliptic estimates, we know $\varphi_n(x)\to \varphi(x),\ in\ C^2_{loc}(\R^2)$ where $\varphi(x)$ satisfies $$-\Delta\varphi(x)=-|x-x_0|^{4N}e^{2\varphi(x)},\ \ in\ \ \R^2,$$ where $\varphi(x)\leq 0$ and $\varphi(0)=0$. This is a contradiction to maximal principle. We proved the claim.

\

Now, we choose $\mu_n^{2(1+N)}\lambda_n e^{\bar{u}_n+\gamma_n}e^{\zeta(x_n)}=1$, i.e. $$\mu_n:=\left(\frac{1}{\lambda_n e^{\bar{u}_n+\gamma_n}e^{\zeta(x_n)}}\right)^{\frac{1}{2(1+N)}}\to 0.$$ It is easy to see that $$\frac{|x_n|}{\mu_n}\leq C$$ and $$\mu_n^{2N}e^{\bar{u}_n+\gamma_n}e^{\zeta(x_n)}\leq C\left(\frac{e^{\bar{u}_n+\gamma_n}}{\lambda_n^N }\right)^{\frac{1}{1+N}}\leq C.$$

W.L.O.G, we assume $$\frac{x_n}{\mu_n}\to -x_0,\ \ \mu_n^{2N}e^{\bar{u}_n+\gamma_n}e^{\psi(x_n)}\to \delta.$$

By a standard elliptic estimate, we know that $$\varphi_n(x)\to \varphi(x)$$ in the sense of $C^2_{loc}(\R^2)$, where
\begin{align*}
-\Delta\varphi(x)=|x-x_0|^{2N}e^{\varphi(x)}\left(1-\delta |x-x_0|^{2N}e^{\varphi(x)}\right),\ \ x\in\R^2,
\end{align*} with $$\delta |x-x_0|^{2N}e^{\varphi(x)}\leq 1$$ and $$\int_{\R^2}|x-x_0|^{2N}e^{\varphi(x)}\left(1-\delta |x-x_0|^{2N}e^{\varphi(x)}\right)dx +\int_{\R^2}|x-x_0|^{2N}e^{\varphi(x)}dx\leq C.$$

By classification result in Lan-Li \cite{Lan-Li}, we know
$$\int_{\R^2}|x-x_0|^{2N}e^{\varphi(x)}\left(1-\delta |x-x_0|^{2N}e^{\varphi(x)}\right)dx =2\pi b,$$ and $$\varphi(x)=-b\log|x|+C+O(|x|^{-\varepsilon}),\ \ |x|\geq 2,$$ where $\varepsilon\in (0,1)$, $b\geq 4(1+N)$ are two constants.

Taking $R_n\to\infty$ such that $$\lim\|\varphi_n(x)-\varphi(x)\|_{L^\infty(B_{2R_n}(0))}=0,$$ which implies that $$w_n(x)+2(1+N)\log|x-x_n|+\log\lambda_ne^{\bar{u}_n}\leq C,\ \ |x-x_n|\leq 2R_n\mu_n.$$ Then, for any $c_0\mu_n\leq |x-x_n|\leq 2R_n\mu_n$ where $c_0>0$ is a positive constant, since $$|x|\leq |x-x_n|+|x_n|\leq (1+\frac{C}{c_0}) |x-x_n|,$$ we get $$w_n(x)+2\log|x-x_n|+2N\log |x|+\log\lambda_ne^{\bar{u}_n}\leq C(c_0).$$

Choosing $x_n^1=x_n$ and $\beta_n^1=R_n\mu_n$, we get the first blow-up point.

\

\textbf{Case 2:} $\lambda_n e^{\bar{u}_n+\gamma_n}|x_n|^{2(1+N)}\to \infty$.

\

In this case, we choose $$\mu_n:=\left( e^{-\gamma_n-\bar{u}_n-\log\lambda_n} |x_n|^{-2N}e^{-\zeta(x_n)}\right)^{\frac{1}{2}},$$ i.e. $\gamma_n=-2\log\mu_n-2N\log |x_n|-\bar{u}_n-\log\lambda_n-\zeta(x_n)$.

It is easy to see that $$\frac{|x_n|}{\mu_n}=\left(\lambda_n e^{\bar{u}_n+\gamma_n}|x_n|^{2(1+N)}e^{\zeta(x_n)}\right)^{\frac{1}{2}}\to\infty.$$

By a standard elliptic estimate, we know that $$\varphi_n(x)\to \varphi(x)$$ in the sense of $C^2_{loc}(\R^2)$, where
\begin{align*}
-\Delta\varphi(x)=e^{\varphi(x)}\left(1-\delta e^{\varphi(x)}\right),\ \ x\in\R^2,
\end{align*} with $$\delta e^{\varphi(x)}\leq 1$$ and $$\int_{\R^2}e^{\varphi(x)}\left(1-\delta |x-x_0|^{2N}e^{\varphi(x)}\right)dx +\int_{\R^2}e^{\varphi(x)}dx\leq C.$$

By classification result in Lan-Li \cite{Lan-Li}, we know
$$\int_{\R^2}e^{\varphi(x)}\left(1-\delta e^{\varphi(x)}\right)dx =2\pi b,$$ and $$\varphi(x)=-b\log|x|+C+O(|x|^{-\varepsilon}),\ \ |x|\geq 2,$$ where $\varepsilon\in (0,1)$, $b\geq 4$ are two constants.

Take $R_n\to \infty$, such that $$R_n\mu_n\to 0,\ \ \frac{|x_n|}{R_n\mu_n}\to\infty,$$ and $$\lim\|\varphi_n(x)-\varphi(x)\|_{L^\infty(B_{2R_n}(0))}=0,$$ which implies  $$w_n(x)+2\log |x-x_n|+2N\log |x_n|+\log\lambda_ne^{\bar{u}_n}\leq C,\ \ |x-x_n|\leq 2R_n\mu_n.$$ This immediately implies
$$w_n(x)+2\log |x-x_n|+2N\log |x|+\log\lambda_ne^{\bar{u}_n}\leq C,\ \ |x-x_n|\leq 2R_n\mu_n,$$ because $$|x|\leq |x-x_n|+|x_n|\leq 2|x_n|.$$

Choosing $x_n^1=x_n$ and $\beta_n^1=R_n\mu_n$, we get the first blow-up point.

\

\textbf{Step 2:} Construction of second blow-up point $x_n^2$ and $\mu_n^2$.

\

If $$w_n(x)+2\log |x-x_n^1|+2N\log |x|+\log\lambda_ne^{\bar{u}_n}\leq C,\ \ \forall \ |x-x_n^1|\geq \beta_n^1,$$ then we have done. Otherwise, we have $$\sup_{\overline{B_1(0)\setminus B_{\beta_n^1}(x_n^1)}}w_n(x)+2\log |x-x_n^1|+2N\log |x|+\log\lambda_ne^{\bar{u}_n}=+\infty.$$ Suppose $y_n$ is a maximal point such that $$w_n(y_n)+2\log |y_n-x_n^1|+2N\log |y_n|+\log\lambda_ne^{\bar{u}_n}=+\infty, \ |y_n-x_n^1|\geq 2R_n\mu_n.$$

Denote $t_n:=\frac{1}{2}|y_n-x_n^1|$ and $$\psi_n(x):=w_n(x)+2\log [t_n-|x-y_n|]+2N\log |x|+\log\lambda_ne^{\bar{u}_n},\ \ x\in B_{t_n}(y_n).$$ Supposing $p_n\in B_{t_n}(y_n)$ such that $\psi_n(p_n)=\sup_{\overline{B_{t_n}(y_n)}}\psi_n(x),$ then $$\psi_n(p_n)\geq \psi_n(y_n)=w_n(y_n)+2\log\frac{1}{2}|y_n-x_n^1|+2N\log |y_n| +\log\lambda_ne^{\bar{u}_n}\to +\infty.$$

Denoting $s_n:=\frac{1}{2}[t_n-|p_n-y_n|]$, then we have
\begin{equation}\label{equ:03}
\psi_n(p_n)=w_n(p_n)+2\log 2s_n+2N\log |p_n| +\log\lambda_ne^{\bar{u}_n}\to +\infty
 \end{equation}
 and for any $x\in B_{s_n}(p_n)$, there holds \begin{align*}
\psi_n(x)&=w_n(x)+2\log [t_n-|x-y_n|]+2N \log |x|+\log\lambda_ne^{\bar{u}_n}\\&\leq \psi_n(p_n)= w_n(p_n)+2\log 2s_n+2N\log |p_n|+\log\lambda_ne^{\bar{u}_n}.\end{align*} Since $$t_n-|x-y_n|\geq s_n,\ \ \forall x\in B_{s_n}(p_n),$$ then we have
\begin{equation}\label{ineq:02}
w_n(x)\leq w_n(p_n)+2N\log\frac{|p_n|}{|x|}+2\log 2,\ \ \forall x\in B_{s_n}(p_n).
\end{equation}

Denote $\tilde{\gamma}_n:=w_n(p_n)$ and $$\varphi_n(x)=w_n(p_n+\tilde{\mu}_nx)-\tilde{\gamma}_n,$$ where $\tilde{\mu}_n$ is a positive number going to zero and satisfying $\frac{s_n}{\tilde{\mu}_n}\to +\infty$, which will be choosed  later.

A direct computation yields
\begin{align*}
-\Delta\varphi_n(x)=\tilde{\mu}_n^{2(1+N)}\lambda_n e^{\bar{u}_n+\tilde{\gamma}_n}|\frac{p_n}{\tilde{\mu}_n}+x|^{2N} e^{\zeta(p_n+\tilde{\mu}_nx)}e^{\varphi_n(x)}
\left(1-\tilde{\mu}_n^{2N}e^{\bar{u}_n+\tilde{\gamma}_n}|\frac{p_n}{\tilde{\mu}_n}+x|^{2N} e^{\zeta(p_n+\tilde{\mu}_nx)}e^{\varphi_n(x)}\right),\ \ in\ \ B_{\frac{s_n}{\tilde{\mu}_n}}(0).
\end{align*}

\

We have to consider the following cases.

\

\textbf{Case 1-1:} If Case 1 happens, we claim first:
\begin{equation}\label{equ:16}
|p_n|e^{\frac{\tilde{\gamma}_n+\bar{u}_n+\log\lambda_n}{2(1+N)}}\to +\infty.
\end{equation}
In fact, since $\frac{|x_n^1|}{\mu_n}\leq C$, we have
\begin{align*}
|p_n|\geq |p_n-x^1_n|-|x_n^1|\geq R_n\mu_n-|x_n^1|\geq \frac{1}{2}R_n\mu_n
\end{align*} and
\begin{align*}
2|s_n|=t_n-|p_n-y_n|=\frac{1}{2}|y_n-x^1_n|-|p_n-y_n|\leq \frac{1}{2}|p_n-x^1_n|\leq \frac{1}{2}|p_n|(1+\frac{|x^1_n|}{|p_n|})\leq |p_n|.
\end{align*} Combing this with the fact \eqref{equ:03} that $$w_n(p_n)+2\log s_n+2N\log |p_n|+\bar{u}_n+\log\lambda_n\to +\infty,$$
we proved \eqref{equ:16}.

We claim secondly:
\begin{equation}\label{ineq:03}
w_n(x)\leq w_n(p_n)+2(1+N)\log 2,\ \ \forall\  x\in B_{s_n}(p_n).
\end{equation} In fact, since $|p_n|\geq 2s_n$ and for any $x\in B_{s_n}(p_n)$, there holds $$\frac{|p_n|}{|x|}\leq \frac{|p_n|}{|p_n|-|x-p_n|}\leq \frac{|p_n|}{|p_n|-s_n}\leq 2,$$ then \eqref{ineq:03} follows immediately from \eqref{ineq:02}.

Now, we take $$\tilde{\mu}_n:=\left( \frac{e^{-\tilde{\gamma}_n-\bar{u}_n-\log\lambda_n}e^{-\zeta(p_n)}}{|p_n|^{2N}} \right)^{\frac{1}{2}},$$ i.e. $$\tilde{\gamma}_n=-2\log\tilde{\mu}_n-2N\log|p_n|-\bar{u}_n-\log\lambda_n-\zeta(p_n).$$ Then \eqref{equ:03} yields $\tilde{\mu}_n\to 0$ and
\begin{align*}
\frac{s_n}{\tilde{\mu}_n}=e^{\frac{1}{2}[\tilde{\gamma}_n+2\log s_n+2N\log |p_n|+\bar{u}_n+\log\lambda_n+\zeta(p_n)]}\to +\infty,
\end{align*} which also implies that $\frac{|p_n|}{\tilde{\mu}_n}\to +\infty$.

By a standard elliptic analysis, we know that $$\varphi_n(x)\to \varphi(x)$$ in the sense of $C^2_{loc}(\R^2)$, where
\begin{align*}
-\Delta\varphi(x)=e^{\varphi(x)}\left(1-\delta e^{\varphi(x)}\right),\ \ x\in\R^2,
\end{align*} with $$\delta e^{\varphi(x)}\leq 1$$ and $$\int_{\R^2}e^{\varphi(x)}\left(1-\delta |x-x_0|^{2N}e^{\varphi(x)}\right)dx +\int_{\R^2}e^{\varphi(x)}dx\leq C.$$

By classification result in Lan-Li \cite{Lan-Li}, we know
$$\int_{\R^2}e^{\varphi(x)}\left(1-\delta e^{\varphi(x)}\right)dx =2\pi b,$$ and $$\varphi(x)=-b\log|x|+C+O(|x|^{-\varepsilon}),\ \ |x|\geq 2,$$ where $\varepsilon\in (0,1)$, $b\geq 4$ are two constants.

Assume $$\varphi(\bar{x})=\max_{\R^2}\varphi(x).$$ Take $R_n\to +\infty$, such that $R_n\tilde{\mu}_n=o(1)s_n$ and $$\|\varphi_n(x)-\varphi(x)\|_{L^\infty(B_{4R_n})}\to 0.$$ Let $q_n\in B_{R_n}(0)$ such that $$\varphi_n(\bar{x}+q_n)=\max_{\overline{B_{R_n}(0)}}\varphi_n(x).$$ Now, setting $$x^2_k:=p_n+\tilde{\mu}_n(\bar{x}+q_n),\ \ \beta_n^2:=R_n\tilde{\mu}_n,$$ one can see that it satisfies all the conditions of proposition.

\

If Case 2 happens, then we have to consider the following two cases:

\

\textbf{Case 2-1:} $\lambda_n e^{\bar{u}_n+\tilde{\gamma}_n}|p_n|^{2(1+N)}\leq C$.

\

Similar to case 1, we first have $$\lambda_n^{-N} e^{\bar{u}_n+\tilde{\gamma}_n}\leq C.$$

In this case, taking $$\tilde{\mu}_n:=\left(\frac{1}{\lambda_n e^{\bar{u}_n+\tilde{\gamma}_n}e^{\zeta(p_n)}}\right)^{\frac{1}{2(1+N)}}\to 0,$$ it is easy to see that $$\frac{|p_n|}{\tilde{\mu}_n}\leq C,\ \ \tilde{\mu}_n^{2N}e^{\bar{u}_n+\tilde{\gamma}_n}\leq C\left(\frac{e^{\bar{u}_n+\tilde{\gamma}_n}}{\lambda_n^N }\right)^{\frac{1}{1+N}}\leq C.$$

Since $$w_n(p_n)+2(1+N)\log |p_n|+\log\lambda_ne^{\bar{u}_n}\leq C,$$ then \eqref{equ:03} yields $$\frac{s_n}{|p_n|}\to \infty$$
and $$\frac{s_n}{\tilde{\mu}_n}=e^{\frac{1}{2(1+N)}\left(w_n(p_n)+2(1+N)\log s_n+\bar{u}_n+\log\lambda_n+\zeta(p_n)\right)}\to\infty.$$

W.L.O.G, we assume $$\frac{p_n}{\tilde{\mu}_n}\to -x_0,\ \ \tilde{\mu}_n^{2N}e^{\bar{u}_n+\tilde{\gamma}_n}e^{\zeta(p_n)}\to \delta.$$

From \eqref{ineq:02}, we see that $$|\Delta \varphi_n(x)|\leq C|\frac{p_n}{\tilde{\mu}_n}+x|^{2N}e^{\varphi_n(x)}\leq C|\frac{p_n}{\tilde{\mu}_n}+x|^{2N}e^{2N\log\frac{p_n}{p_n+\tilde{\mu}_nx}}\leq C\frac{|p_n|^{2N}}{\tilde{\mu}_n^{2N}}\leq C.$$

Then a standard elliptic estimate yields $\varphi_n(x)\to \varphi(x)$ in the sense of $C^2_{loc}(\R^2)$, where $\varphi(x)$ satisfies \eqref{equ:04}. By classification result in Lan-Li \cite{Lan-Li}, we have
 \begin{equation}\label{equ:18}
 \varphi(x)=-b\log|x|+C+O(|x|^{-\varepsilon}),\ \ |x|\geq 2,\end{equation} where $\varepsilon\in (0,1)$, $b\geq 4(N+1)$ are two constants.

Taking $R_n\to\infty$ such that $R_n\tilde{\mu}_n=o(1)s_n$ and $$\lim\|\varphi_n(x)-\varphi(x)\|_{L^\infty(B_{4R_n}(0))}=0,$$ which implies that $$w_n(x)+2(1+N)\log|x-p_n|+\log\lambda_ne^{\bar{u}_n}\leq C,\ \ |x-p_n|\leq 4R_n\tilde{\mu}_n.$$ Then, for any $c_0\tilde{\mu}_n\leq |x-p_n|\leq 4R_n\tilde{\mu}_n$ where $c_0>0$ is a positive constant, since $$|x|\leq |x-p_n|+|p_n|\leq |x-p_n|+C\tilde{\mu}_n\leq (1+\frac{1}{C c_0}) |x-p_n|,$$ we get $$w_n(x)+2\log|x-p_n|+2N\log |x|+\log\lambda_ne^{\bar{u}_n}\leq C.$$

Similarly to Case 1-1, we can take $x^2_n$ and $\beta^2_n$.

\

\textbf{Case 2-2:} $\lambda_n e^{\bar{u}_n+\tilde{\gamma}_n}|p_n|^{2(1+N)}\to\infty$.

\

In this case, we choose $$\tilde{\mu}_n:=\left( e^{-\tilde{\gamma}_n-\bar{u}_n-\log\lambda_n} |p_n|^{-2N}e^{-\zeta(p_n)}\right)^{\frac{1}{2}},$$ i.e. $\tilde{\gamma}_n=-2\log\tilde{\mu}_n-2N\log |p_n|-\bar{u}_n-\log\lambda_n-\zeta(p_n)$.

It is easy to see that $$\frac{|p_n|}{\tilde{\mu}_n}=\left(\lambda_n e^{\bar{u}_n+\tilde{\gamma}_n}|p_n|^{2(1+N)}e^{\zeta(p_n)}\right)^{\frac{1}{2}}\to\infty.$$

W.L.O.G, we may assume $s_n\leq\frac{1}{2}|p_n|$. Otherwise, we take $\tilde{s}_n=\min\{s_n,\frac{1}{2}|p_n|\}$ and consider in $B_{\tilde{s}_n}(p_n)$. Then we have $$\frac{s_n}{\tilde{\mu}_n}=e^{\frac{1}{2}\left( w_n(p_n)+2\log s_n+2N\log |p_n|+\bar{u}_n+\log\lambda_n+\zeta(p_n)\right)}\to\infty,$$ where the last limit follows from \eqref{equ:03}.

Since for any $x\in B_{s_n}(p_n)$, there holds $|x|\geq \frac{1}{2}|p_n|$. By \eqref{ineq:02}, we have $$w_n(x)\leq w_n(p_n)+C,\ \ \forall \ x\in B_{s_n}(p_n).$$

By a standard elliptic estimate, we know that $$\varphi_n(x)\to \varphi(x)$$ in the sense of $C^2_{loc}(\R^2)$, where $\varphi(x)$ satisfies \eqref{equ:07}. By classification result in Lan-Li \cite{Lan-Li}, we know $$\varphi(x)=-b\log|x|+C+O(|x|^{-\varepsilon}),\ \ |x|\geq 2,$$ where $\varepsilon\in (0,1)$, $b\geq 4$ are two constants. Similarly to Case 1-1, we can take $x^2_n$ and $\beta^2_n$.

\

\textbf{Step 3:} By above two steps, we have defined the selection process. Continuously, we consider the function $$w_n(x)+2\log dist(x,\{x_n^1,x_n^2\})+2N\log |x|+\log\lambda_ne^{\bar{u}_n}\leq C,\ \ \forall \ |x-x_n^i|\geq \beta_n^i,\ i=1,2.$$ If it is uniformly bounded, then we stop and conclude the proposition for $m=2$. Otherwise, using the same argument, we get $x_n^3$ and $\beta_n^3$. Since each bubble area contributes a positive energy, this process must stop after finite steps. We proved the proposition.

\end{proof}

\

\begin{defn}
For $x_n^i\in\Sigma_n$, we call $(x_n^i,\mu_n^i)$ a bubble. Specially, we call it a singular bubble if the first case of conclusion $(4)$ in Proposition \ref{prop:bubble-tree} happens.
\end{defn}

\

By using the singularities we got in Proposition \ref{prop:bubble-tree} and a blow-up argument, we prove following the oscillation estimate.
\begin{lem}\label{lem:osc}
Let $w_n$ be a sequence of solutions of \eqref{equ:02} satisfying \eqref{equ:06}, \eqref{equ:08} and \eqref{equ:15}.  Denote $$osc_{\Omega}u:=\sup_{x,y\in\Omega}(u(x)-u(y))$$ and $d_n(x):=dist(x,\Sigma_n)$. Then we have:
 $$osc_{\frac{1}{2}d_n(x)}w_n\leq C,\ \ \forall \ x\in B\setminus \Sigma_n,$$  where $C$ is a constant independent of $x,\ n$.
\end{lem}

\begin{proof}
Since $0$ is the only blow-up point in $B_1(0)$ and $osc_{\partial B}w_n\leq C$, by a standard estimate, we know $$osc_{B\setminus B_\delta}w_n\leq C(\delta),$$ for any $\delta\in (0,1)$. Hence we just need to show the lemma in $B_{\frac{1}{20}}$. Let $\phi_n(x)$ be the solution of
\begin{align*}
\begin{cases}
-\Delta \phi_n(x)=0,\ \ &in\ \ B_1(0),\\
\phi_n(x)=w_n(x),\ \ &on\ \ \partial B_1(0).
\end{cases}
\end{align*}
It is clear that $$osc_{B}\phi_n(x)\leq osc_{\partial B}w_n(x) \leq C.$$

Let $\eta_n(x):=w_n(x)-\phi_n(x)$ and let $$G(x,y)=-\frac{1}{2\pi}\log |x-y|+H(x,y)$$ be Green's function on $B_1$ with respect to the Dirichlet boundary, where $H(x,y)$ is a smooth harmonic function. Then $$\eta_n(x)=\int_{B}G(x,y)(-\Delta w_n(y))dy.$$ Let $x_0\in B_{\frac{1}{20}}$ and $r_n:=dist(x_0,\Sigma_n)$. For any $x_1,x_2\in B_{\frac{1}{2}r_n}(x_0)$, then
\begin{align*}
\eta_n(x_1)-\eta_n(x_2)&=\int_{B}\left(G(x_1,y)-G(x_2,y)\right)(-\Delta w_n(y))dy\\
&=\frac{1}{2\pi}\int_{B}\log\frac{|x_1-y|}{|x_2-y|}\Delta w_n(y)dy+O(1).
\end{align*}

We divide the integral into two parts, i.e. $B_1=B_{\frac{3}{4}r_n}(x_0)\cup (B_1\setminus B_{\frac{3}{4}r_n}(x_0))$. Noting that $$\left|\log\frac{|x_1-y|}{|x_2-y|}\right|\leq C,\ \ y\in B_1\setminus B_{\frac{3}{4}r_n}(x_0),$$ we have $$\left|\int_{B_1\setminus B_{\frac{3}{4}r_n}(x_0)}\log\frac{|x_1-y|}{|x_2-y|}\Delta w_n(y)dy \right|\leq C\int_{B_1}|\Delta w_n(y)|dy \leq C.$$

A direct computation yields
\begin{align*}
&\left|\int_{B_1\setminus B_{\frac{3}{4}r_n}(x_0)}\log\frac{|x_1-y|}{|x_2-y|}\Delta w_n(y)dy \right|\\&=\left|\int_{ B_{\frac{3}{4}r_n}(x_0)}\log\frac{|x_1-y|}{|x_2-y|}\lambda_n e^{\bar{u}_n}|y|^{2N}e^{\zeta(y)} e^{w_n(y)}(1-e^{\bar{u}_n} |y|^{2N}e^{\zeta(y)} e^{w_n(y)})dy \right|\\&\leq C\int_{ B_{\frac{3}{4}r_n}(x_0)} \left|\log\frac{|x_1-y|}{|x_2-y|}\right|\lambda_n e^{\bar{u}_n}|y|^{2N} e^{w_n(y)}dy\\
&=C\int_{ B_{\frac{3}{4}}(0)}\left|\log\frac{|x_1-x_0-r_ny|}{|x_2-x_0-r_ny|}\right|\lambda_n e^{\bar{u}_n}|x_0+r_ny|^{2N} e^{w_n(x_0+r_ny)}r_n^2dy\\
&=C\int_{ B_{\frac{3}{4}}(0)}\left|\log\frac{|\frac{x_1-x_0}{r_n}-y|}{|\frac{x_2-x_0}{r_n}-y|}\right|\lambda_n e^{\bar{u}_n}|x_0+r_ny|^{2N} e^{w_n(x_0+r_ny)}r_n^2dy.
\end{align*}

\

Now, we have to consider the following three cases:

\

\textbf{Case 1:} if there is no singular bubble or there is a singular bubble $(x_n^i,\mu_n^i)$ such that $r_n=dist(x_0,\Sigma_n)=|x_0-x_n^i|\geq \mu_n^i$.

\

In this case, we claim:
\begin{equation}\label{ineq:01}
w_n(x_0+r_ny)+2\log r_n+2N\log |x_0+r_ny|+\log\lambda_ne^{\bar{u}_n}\leq C,\ \ \forall \ y\in B_{\frac{3}{4}}(0).
\end{equation}

In fact, if there is no singular bubble, since
\begin{equation}\label{ineq:04}
dist(x_0+r_ny,\Sigma_n)\geq dist(x_0,\Sigma_n)-dist(x_0,x_0+r_ny)\geq \frac{1}{4}r_n,\ \ \forall\ y\in B_{\frac{3}{4}}(0),\end{equation} then \eqref{ineq:01} follows immediately from Proposition \ref{prop:bubble-tree}.

If there is a singular bubble $(x_n^i,\mu_n^i)$, by Proposition \ref{prop:bubble-tree}, \eqref{ineq:04} and the fact that
\begin{equation}\label{equ:05}
dist(x_0+r_ny,x^i_n)\geq dist(x_0+r_ny,\Sigma_n)\geq \frac{1}{4}r_n=\frac{1}{4} dist(x_0,\Sigma_n)\geq \frac{1}{4}\mu_n^i,\ \ \forall\ y\in B_{\frac{3}{4}}(0),\end{equation} we get \eqref{ineq:01}.

Therefore, we get $$\left|\int_{B_1\setminus B_{\frac{3}{4}r_n}(x_0)}\log\frac{|x_1-y|}{|x_2-y|}\Delta w_n(y)dy \right|\leq \int_{ B_{\frac{3}{4}}(0)}\left|\log\frac{|\frac{x_1-x_0}{r_n}-y|}{|\frac{x_2-x_0}{r_n}-y|}\right|dy\leq C,$$ which implies $$|\eta_n(x_1)-\eta_n(x_2)|\leq C.$$ Then the conclusion of the lemma follows.

\

\textbf{Case 2:} if there is a singular bubble $(x_n^i,\mu_n^i)$ such that $r_n=dist(x_0,\Sigma_n)=|x_0-x_n^i|\leq \mu_n^i$.

\

In this case, there must be that $$r_n=dist(x_0,\Sigma_n)=|x_0-x_n^i|.$$  By  Proposition \ref{prop:bubble-tree}, we know $$\varphi_n(x)=w_n(x_n^i+\mu_n^ix)-w_n(x_n^i)\to \varphi(x)$$ in the sense of $C^2_{loc}(\R^2)$, where $\varphi(x)$ satisfies $$|\varphi(x)|\leq C+b\log (1+|x|).$$ It is easy to see that, when  $|x_0-x_n^i|\leq \mu_n^i$, there holds $$osc_{\frac{1}{2}d_n(x_0)}w_n(x)=osc_{B_{\frac{|x_0-x_n^i|}{2}}(x_0)}w_n(x)=osc_{B_{\frac{|x_0-x_n^i|}{2\mu_n^i}}(\frac{x_0-x_n^i}{\mu_n^i})}\varphi_n(x)\leq osc_{B_2(0)}\varphi_n(x)\leq C,$$ which implies immediately the conclusion of the lemma.

\

\textbf{Case 3:} if there is a singular bubble $(x_n^i,\mu_n^i)$ and $r_n=dist(x_0,\Sigma_n)=|x_0-x_n^j|$ where $j\neq i$.

\

Then we  have the two following cases.

$(i)$ $\frac{r_n}{\mu_n^i}\to +\infty$. In this case, it is easy to see that \eqref{equ:05} still holds and then the conclusion of the lemma follows as in Case 1.

$(i)$ $\frac{r_n}{\mu_n^i}\leq C$. In this case, by proposition \ref{prop:bubble-tree}, we know $\frac{|x_n^i-x_n^j|}{\mu_n^i}\to +\infty$, which implies $$dist(x_0+r_ny,x_n^i)\geq |x_0-x_n^i|-\frac{3}{4}r_n\geq |x_n^i-x_n^j|-|x_0-x_n^j|-\frac{3}{4}r_n=|x_n^i-x_n^j|-\frac{7}{4}r_n\geq \mu_n^i.$$ Then the conclusion of the lemma follows as in Case 1.

We proved this lemma.
\end{proof}

\

Before stating the next lemma, we first extend definitions of fast decay and slow decay of \cite{Lin-Wei-Zhang} for Liouville type system to Chern-Simon-Higgs equation.
\begin{defn}
We say $w_n$ has fast decay on $\partial B_{r_n}(x_n)$ (resp. $B_{r_n}(x_n)\setminus B_{s_n}(x_n)$) if $$w_n(x)+2\log dist(x,\Sigma_n)+ 2N\log |x|+\log\lambda_ne^{\bar{u}_n}\leq -N_n,\ \forall x\in \partial B_{r_n}(x_n)\ \ (resp. \ \forall x\in B_{r_n}(x_n)\setminus B_{s_n}(x_n)),$$ for some $N_n\to\infty$ as $n\to\infty$.
\end{defn}

\

Next, we derive a fast decay lemma. For classical Liouville equation or Toda system without singular sources (i.e. $N=0$), it is a direct corollary of oscillation Lemma \ref{lem:osc}, see \cite{Lin-Wei-Zhang}. Here, in our case, we need to deal with the singular term $2N\log |x|$.
\begin{lem}\label{lem:fast-decay}
Let $\Sigma^1_n\subset \Sigma_n$ be a subset of $\Sigma_n$ with $$\Sigma^1_n\subset B_{r_n}(x_n)\subset B_1(0).$$ Suppose $$dist(\Sigma^1_n,\partial B_{r_n}(x_n))=o(1)dist(\Sigma_n\setminus \Sigma^1_n,\partial B_{r_n}(x_n)).$$ Then for any $s_n\geq 2r_n$ with $s_n=o(1)dist(\Sigma_n\setminus \Sigma^1_n,\partial B_{r_n}(x_n))$, if $u_n$ has fast decay on $\partial B_{s_n}(x_n)$, then for any $\beta_n\to\infty$ with $\beta_ns_n=o(1)dist(\Sigma_n\setminus \Sigma^1_n,\partial B_{r_n}(x_n))$, there exists $\alpha_n\to\infty$ with $\alpha_n=o(1)\beta_n$ such that $w_n$ has fast decay in $B_{\alpha_ns_n}\setminus B_{s_n}$, i.e. $$w_n(x)+2\log dist(x,\Sigma_n)+ 2N\log |x|+\log\lambda_ne^{\bar{u}_n}\leq -N_n, \ \forall s_n\leq |x-x_n|\leq \alpha_ns_n,$$ for some $N_n\to\infty$, and $$\int_{B_{\alpha_ns_n}(x_n)\setminus B_{s_n}(x_n)}\lambda_n e^{\bar{u}_n}|x|^{2N}e^{\zeta(x)}e^{w_n(x)}dx=o(1).$$
\end{lem}

\begin{proof}

Since $w_n(x)$ has fast decay on $\partial B_{s_n}(x_n)$, we have  $$w_n(x)+2\log dist(x,\Sigma_n)+ 2N\log |x|+\log\lambda_ne^{\bar{u}_n}\leq -N_n, \ \forall  |x-x_n|=s_n,$$ for some $N_n\to +\infty$.

Next, we claim: for any fixed $\Lambda>0$, there holds
\begin{equation}\label{ineq:08}
w_n(x)+2\log dist(x,\Sigma_n)+ 2N\log |x|+\log\lambda_ne^{\bar{u}_n}\leq -N_n+C(\Lambda), \ \forall  s_n\leq |x-x_n|\leq \Lambda s_n.
\end{equation}

In fact, by Lemma \ref{lem:osc}, for any $\Lambda>1$, we get
\begin{equation}\label{ineq:05}
osc_{B_{\Lambda s_n}(x_n)\setminus B_{ s_n}(x_n) }w_n(x)\leq C(\Lambda).
 \end{equation}
 Noting that for any $x\in B_{\Lambda s_n}(x_n)\setminus B_{ s_n}(x_n) $, there hold
\begin{align}
dist(x,\Sigma_n)&\leq |x-x_n|+dist(x_n,\Sigma_n)\leq (\Lambda+\frac{1}{2})s_n,\label{ineq:06}\\
dist(x,\Sigma_n)&\geq s_n-r_n\geq \frac{1}{2}s_n,\label{ineq:07}
 \end{align} which implies
 \begin{equation}\label{ineq:11}
osc_{B_{\Lambda s_n}(x_n)\setminus B_{ s_n}(x_n) }\log dist(x,\Sigma_n)\leq C(\Lambda).
 \end{equation}

\

Now, we consider the following two cases:

\

\textbf{Case 1:} $|x_n|\geq 2 s_n$.

\

In this case, for any $x\in B_{\Lambda s_n}(x_n)\setminus B_{ s_n}(x_n) $, there holds
\begin{align*}
|x|&\leq |x-x_n|+|x_n|\leq \Lambda s_n+|x_n|\leq (1+\frac{\Lambda}{2})|x_n|.
 \end{align*}

Taking any point $p_n\in\partial B_{s_n}(x_n)$, since $|p_n|\geq |x_n|-|p_n-x_n|= |x_n|-s_n\geq \frac{1}{2}|x_n|$, thus, we obtain \begin{align*}
&w_n(x)+2\log dist(x,\Sigma_n)+ 2N\log |x|+\log\lambda_ne^{\bar{u}_n}\\
&\leq w_n(p_n)+2\log dist(p_n,\Sigma_n)+ 2N\log |p_n|+\log\lambda_ne^{\bar{u}_n}+2N\log\frac{|x|}{|p_n|}+C(\Lambda)\\&\leq -N_k+C(\Lambda), \ \forall s_n\leq |x-x_n|\leq \Lambda s_n,
\end{align*} which implies \eqref{ineq:08} immediately.

\

\textbf{Case 2:} $|x_n|\leq 2s_n$.

\

In this case, we first choose a point $y_n\in\partial B_{s_n}(x_n)$ such that $|y_n|\geq s_n$. Then we have $$w_n(y_n)+2\log dist(y_n,\Sigma_n)+ 2N\log |y_n|+\log\lambda_ne^{\bar{u}_n}\leq -N_k.$$

Noting that for any $x\in B_{\Lambda s_n}(x_n)\setminus B_{ s_n}(x_n) $, if $|x|\leq \frac{1}{2}s_n$, then by \eqref{ineq:05}-\eqref{ineq:07}, we have
\begin{align*}
&w_n(x)+2\log dist(x,\Sigma_n)+ 2N\log |x|+\log\lambda_ne^{\bar{u}_n}\\&\leq
w_n(y_n)+2\log dist(y_n,\Sigma_n)+ 2N\log |y_n|+\log\lambda_ne^{\bar{u}_n}\\&\leq -N_k+C(\Lambda).
\end{align*}

For $x\in B_{\Lambda s_n}(x_n)\setminus B_{ s_n}(x_n) $, if $|x|\geq \frac{1}{2}s_n$, since $$\frac{1}{2}s_n\leq |x|\leq |x-x_n|+|x_n|\leq (\Lambda+2)s_n,$$ we get
\begin{align*}
osc_{x\in B_{\Lambda s_n}(x_n)\setminus B_{ s_n}(x_n),\ |x|\geq \frac{1}{2}s_n}\log |x|\leq C(\Lambda).
\end{align*}

Combing this with \eqref{ineq:05}-\eqref{ineq:07}, we have
\begin{align*}
&w_n(x)+2\log dist(x,\Sigma_n)+ 2N\log |x|+\log\lambda_ne^{\bar{u}_n}\\&\leq
w_n(y_n)+2\log dist(y_n,\Sigma_n)+ 2N\log |y_n|+\log\lambda_ne^{\bar{u}_n}\\&\leq -N_k+C(\Lambda).
\end{align*} Thus, we proved \eqref{ineq:08}.

By \eqref{ineq:08}, it is easy to see that, for any $\beta_n\to\infty$, there exists $\alpha^1_n\to\infty,\ \alpha^1_n=o(1)\beta_n$ such that
\begin{equation}\label{ineq:09}
w_n(x)+2\log dist(x,\Sigma_n)+ 2N\log |x|+\log\lambda_ne^{\bar{u}_n}\leq -\frac{1}{2}N_n, \ \forall  s_n\leq |x-x_n|\leq \alpha^1_n s_n.
\end{equation}
Now, we can choose $\alpha_n\to\infty,\ \ \alpha_n=o(1)\alpha^1_n$ and $e^{-\frac{1}{2}N_n}\log\alpha_n=o(1)$. Then we have
\begin{align*}
\int_{B_{\alpha_ns_n}(x_n)\setminus B_{s_n}(x_n)}\lambda_n e^{\bar{u}_n}|x|^{2N}e^{\zeta(x)}e^{w_n(x)}dx&\leq C e^{-\frac{1}{2}N_n} \int_{B_{\alpha_ns_n}(x_n)\setminus B_{s_n}(x_n)}e^{-2\log dist(x,\Sigma_n)}dx\\&\leq C e^{-\frac{1}{2}N_n+2\log 2} \int_{B_{\alpha_ns_n}(x_n)\setminus B_{s_n}(x_n)}e^{-2\log |x-x_n|}dx\\&\leq Ce^{-\frac{1}{2}N_n+2\log 2}\log\alpha_n=o(1),
\end{align*} where the second inequality follows from the fact that $$ dist(x,\Sigma_n^1)\geq |x-x_n|- dist(x_n,\Sigma^1_n)\geq |x-x_n|-r_n\geq |x-x_n|-\frac{1}{2}s_n\geq \frac{1}{2}|x-x_n|.$$ We prove the lemma.

\end{proof}

\

\begin{lem}\label{lem:01}
Let $\Sigma^1_n\subset \Sigma_n$ be a subset of $\Sigma_n$ with $$\Sigma^1_n\subset B_{r_n}(x_n)\subset B_1(0).$$ Suppose $$dist(\Sigma^1_n,\partial B_{r_n}(x_n))=o(1)dist(\Sigma_n\setminus \Sigma^1_n,\partial B_{r_n}(x_n)).$$

For any $s_n\geq 2r_n$ and $R_n\to\infty$ satisfying $$R_ns_n=o(1)dist(\Sigma_n\setminus \Sigma^1_n,\partial B_{r_n}(x_n)),$$ if $w_n$ has fast decay in $ B_{R_ns_n}(x_n)\setminus B_{s_n}(x_n)$ and \begin{equation}\label{equ:12}
\int_{B_{R_ns_n}(x_n)\setminus B_{s_n}(x_n)}\lambda_ne^{\bar{u}_n}|x|^{2N}e^{w_n}dx=o(1),
\end{equation} then we have $$\nabla w_n(x)=-\frac{x-x_n}{|x-x_n|^2}\alpha_n(s_n;x_n)+\frac{o(1)}{|x-x_n|},\ \ x\in\partial B_{\sqrt{R_n}s_n}(x_n).$$
\end{lem}
\begin{proof}
The proof of this lemma is more or less standard. Here, for reader's convenience, we give a detailed proof.

Let $\phi_n(x)$ be the solution of
\begin{align*}
\begin{cases}
-\Delta \phi_n(x)=0,\ \ &in\ \ B_1(0),\\
\phi_n(x)=w_n(x),\ \ &on\ \ \partial B_1(0).
\end{cases}
\end{align*}
A standard elliptic theory yields that $\|w_n(x)\|_{C^1(B_1\setminus B_{\delta})}\leq C(\delta)$ for any $\delta>0$. Then it is clear that $$\|\phi_n(x)\|_{C^1(B_1(0))}\leq \|w_n(x)\|_{C^1(\partial B)} \leq C.$$ Let $\eta_n(x):=w_n(x)-\phi_n(x)$, then $$\eta_n(x)=\int_{B}G(x,y)(-\Delta w_n(y))dy,$$ where $G(x,y)$ is the Green function as in Lemma \ref{lem:osc}. Thus,
\begin{equation*}
\nabla \eta_n(x)=-\frac{1}{2\pi}\int_{B}\frac{x-y}{|x-y|^2}(-\Delta w_n(y))dy.
\end{equation*}

Next, we divide the integral domain $B_1(0)$ as follows:
\begin{align*}
B_1(0)=&B_1(0)\setminus B_{R_ns_n}(x_n)\cup B_{R_ns_n}(x_n)\setminus B_{\frac{3}{2}\sqrt{R_n}s_n}(x_n)\cup B_{\frac{3}{2}\sqrt{R_n}s_n}(x_n)\setminus B_{\frac{1}{2}\sqrt{R_n}s_n}(x_n)\\&\cup B_{\frac{1}{2}\sqrt{R_n}s_n}(x_n)\setminus B_{s_n}(x_n)\cup B_{s_n}(x_n).\end{align*}

For $x\in \partial B_{\sqrt{R_n}s_n}(x_n)$ and $y\in B_1(0)\setminus B_{R_ns_n}(x_n)$, since $$\frac{x-y}{|x-y|^2}=\frac{o(1)}{|x-x_n|},$$ then we immediately have $$-\frac{1}{2\pi}\int_{B_1(0)\setminus B_{R_ns_n}(x_n)}\frac{x-y}{|x-y|^2}(-\Delta w_n(y))dy=\frac{o(1)}{|x-x_n|}.$$

For $y\in B_{R_ns_n}(x_n)\setminus B_{\frac{3}{2}\sqrt{R_n}s_n}(x_n)$, since $$|\frac{x-y}{|x-y|^2}|\leq \frac{C}{|x-x_n|},$$ by \eqref{equ:12}, we have $$\left|\frac{1}{2\pi}\int_{B_1(0)\setminus B_{R_ns_n}(x_n)}\frac{x-y}{|x-y|^2}(-\Delta w_n(y))dy\right|\leq  \frac{C}{|x-x_n|}\int_{B_1(0)\setminus B_{R_ns_n}(x_n)}(-\Delta w_n(y))dy=\frac{o(1)}{|x-x_n|}.$$

For $y\in B_{\frac{3}{2}\sqrt{R_n}s_n}(x_n)\setminus B_{\frac{1}{2}\sqrt{R_n}s_n}(x_n)$, noting that $$B_{\frac{3}{2}\sqrt{R_n}s_n}(x_n)\setminus B_{\frac{1}{2}\sqrt{R_n}s_n}(x_n)\subset \bigcup_{j=1}^JB_{\frac{|x-x_n|}{2}}(z_j),$$ where $z_j\in \partial B_{\sqrt{R_n}s_n}(x_n)$, $J$ is a universal constant independent of $n$, since $w_n$ has fast decay in $B_{R_ns_n}(x_n)\setminus B_{s_n}(x_n)$, we have $$|\Delta w_n(x)|\leq C \lambda_ne^{\bar{u}_n}|x|^{2N}e^{w_n}\leq\frac{o(1)}{|x-x_n|^2}.$$ Then,
\begin{align*}
\left|\frac{1}{2\pi}\int_{B_1(0)\setminus B_{R_ns_n}(x_n)}\frac{x-y}{|x-y|^2}(-\Delta w_n(y))dy\right|\leq \frac{o(1)}{|x-x_n|^2}\sum_{j=1}^J\int_{B_{\frac{|x-x_n|}{2}}(z_j)}\frac{1}{|x-y|}dy=\frac{o(1)}{|x-x_n|}.
\end{align*}

For $y\in B_{\frac{1}{2}\sqrt{R_n}s_n}(x_n)\setminus B_{s_n}(x_n)$, since $$|\frac{x-y}{|x-y|^2}|\leq \frac{C}{|x-x_n|},$$ by \eqref{equ:12}, we have $$\left|\frac{1}{2\pi}\int_{B_1(0)\setminus B_{R_ns_n}(x_n)}\frac{x-y}{|x-y|^2}(-\Delta w_n(y))dy\right|\leq  \frac{C}{|x-x_n|}\int_{B_1(0)\setminus B_{R_ns_n}(x_n)}(-\Delta w_n(y))dy=\frac{o(1)}{|x-x_n|}.$$

For $y\in B_{s_n}(x_n)$, since $$\frac{x-y}{|x-y|^2}=\frac{x-x_n}{|x-x_n|^2}+\frac{o(1)}{|x-x_n|},$$ we have  $$-\frac{1}{2\pi}\int_{B_{s_n}(x_n)}\frac{x-y}{|x-y|^2}(-\Delta w_n(y))dy=-\frac{x-x_n}{|x-x_n|^2}\alpha(s_n;x_n)+\frac{o(1)}{|x-x_n|}.$$ This immediately implies the conclusion of lemma.
\end{proof}

\

Now, at this position, we derive two kinds of  Pohozaev identities (local type and global type) for Chern-Simon-Higgs equation which play an important role in our proof.
\begin{lem}\label{lem:pohozaev}
Let $\Sigma^1_n\subset \Sigma_n$ be a subset of $\Sigma_n$ with $$\Sigma^1_n\subset B_{r_n}(x_n)\subset B_1(0).$$ Suppose $$dist(\Sigma^1_n,\partial B_{r_n}(x_n))=o(1)dist(\Sigma_n\setminus \Sigma^1_n,\partial B_{r_n}(x_n)).$$

Then for any $s_n\geq 2r_n$ with $s_n=o(1)dist(\Sigma_n\setminus \Sigma^1_n,\partial B_{r_n}(x_n))$, if $u_n$ has fast decay on $\partial B_{s_n}(x_n)$, we have the  following Pohozaev identities:
\begin{itemize}
\item[(1)] local type: if $\lim\frac{|x_n|}{s_n}=\infty$, then $$\int_{B_{s_n}(x_n)}2\lambda_ne^{\bar{u}_n}|x|^{2N}e^{\zeta(x)}e^{w_n}[1-\frac{1}{2}e^{\bar{u}_n}|x|^{2N}e^{\zeta(x)}e^{w_n}]dx=\pi\alpha^2_n(s_n;x_n)+o(1).$$
\item[(2)] global type: if $\frac{|x_n|}{s_n}\leq C$, then $$\int_{B_{s_n}(x_n)}2\lambda_ne^{\bar{u}_n}|x|^{2N}e^{\zeta(x)}e^{w_n}[1-\frac{1}{2}e^{\bar{u}_n}|x|^{2N}e^{\zeta(x)}e^{w_n}]dx=\pi\alpha^2_n(s_n;x_n)-4\pi N\alpha_n(s_n;x_n)+o(1).$$
\end{itemize}

\end{lem}
\begin{proof}
We divide the proof into two steps according two cases in the lemma.

\

\textbf{Step 1:} If $\lim\frac{|x_n|}{s_n}=\infty$.

\

By Lemma \ref{lem:fast-decay}, we know there exists $$R_n\to\infty,\ \ R_ns_n=o(1)|x_n|,\ \ R_ns_n=o(1)dist(\Sigma_n\setminus \Sigma^1_n,\partial B_{r_n}(x_n)),$$ such that $w_n$ has fast decay in $B_{R_ns_n}(x_n)\setminus B_{s_n}(x_n)$, i.e. $$w_n(x)+2\log dist(x,\Sigma_n)+ 2N\log |x|+\log\lambda_ne^{\bar{u}_n}\leq -N_n, \ \forall s_n\leq |x-x_n|\leq R_ns_n,$$ for some $N_n\to\infty$ and
\begin{equation}\label{equ:10}
\int_{B_{R_ns_n}(x_n)\setminus B_{s_n}(x_n)}\lambda_n e^{\bar{u}_n}|x|^{2N}e^{\zeta(x)}e^{w_n(x)}dx=o(1).
\end{equation}

Multiplying \eqref{equ:02} by $(x-x_n)\nabla w_n$ and integrating by parts, we get the following Pohozaev type identity that
\begin{align}\label{equ:pohozaev-1}
\int_{\partial B_{s}(x_n)}r\left(\frac{1}{2}|\nabla w_n|^2-|\frac{\partial w_n}{\partial r}|^2\right)d\theta=& \int_{\partial B_{s}(x_n)}r\lambda_ne^{\bar{u}_n}|x|^{2N}e^{\zeta}e^{w_n}[1-\frac{1}{2}e^{\bar{u}_n}|x|^{2N}e^{\zeta}e^{w_n}]d\theta\notag\\ &- \int_{B_{s}(x_n)}2\lambda_ne^{\bar{u}_n}|x|^{2N}e^{\zeta}e^{w_n}[1-\frac{1}{2}e^{\bar{u}_n}|x|^{2N}e^{\zeta}e^{w_n}]dx\notag\\
& - \int_{B_{s}(x_n)}2N\lambda_ne^{\bar{u}_n}|x|^{2N-2}(x\cdot (x-x_n) )e^{\zeta}e^{w_n}[1-e^{\bar{u}_n}|x|^{2N}e^{\zeta}e^{w_n}]dx\notag\\
& - \int_{B_{s}(x_n)}\lambda_ne^{\bar{u}_n}|x|^{2N}( (x-x_n)\cdot \nabla \zeta )e^{\zeta}e^{w_n}[1-e^{\bar{u}_n}|x|^{2N}e^{\zeta}e^{w_n}]dx
\end{align} for any $s_n\leq s\leq R_ns_n$, where $(r,\theta)$ is the polar coordinate system centered at point $x_n$.

Now, we take $s=\sqrt{R_n}s_n$ in \eqref{equ:pohozaev-1}. Since $w_n$ has fast decay on $\partial B_{\sqrt{R_n}s_n}(x_n)$, we have
\begin{align*}
\int_{\partial B_{s}(x_n)}r\lambda_ne^{\bar{u}_n}|x|^{2N}e^{\zeta}e^{w_n}[1-\frac{1}{2}e^{\bar{u}_n}|x|^{2N}e^{\zeta}e^{w_n}]\leq C \int_{\partial B_{s}(x_n)}r\lambda_ne^{\bar{u}_n}|x|^{2N}e^{w_n}\leq Ce^{-N_n}=o(1).
\end{align*}

Noting that for any $x\in B_{\sqrt{R_n}s_n}(x_n)$, there holds $$|x-x_n|\leq \sqrt{R_n}s_n\leq o(1)|x_n|\leq o(1)(|x|+|x-x_n|),$$ which implies $$|x-x_n|=o(1)|x|.$$ Then
\begin{align*}
&\left|\int_{B_{s}(x_n)}2N\lambda_ne^{\bar{u}_n}|x|^{2N-2}(x\cdot (x-x_n) )e^{\zeta}e^{w_n}[1-\frac{1}{2}e^{\bar{u}_n}|x|^{2N}e^{\zeta}e^{w_n}]dx\right|\notag\\
&\leq o(1)\int_{B_{s}(x_n)}2N\lambda_ne^{\bar{u}_n}|x|^{2N}e^{\zeta}e^{w_n}dx=o(1).
\end{align*}

Since $|\nabla \zeta|\leq C$, then $$\left|\int_{B_{s}(x_n)}\lambda_ne^{\bar{u}_n}|x|^{2N}( (x-x_n)\cdot \nabla \zeta )e^{\zeta}e^{w_n}[1-e^{\bar{u}_n}|x|^{2N}e^{\zeta}e^{w_n}]dx\right|\leq Cs.$$

Combining these together, by \eqref{equ:pohozaev-1}, we arrived at
\begin{align}
&\int_{\partial B_{\sqrt{R_n}s_n}(x_n)}r\left(\frac{1}{2}|\nabla w_n|^2-|\frac{\partial w_n}{\partial r}|^2\right)d\theta\notag\\&= - \int_{B_{\sqrt{R_n}s_n}(x_n)}2\lambda_ne^{\bar{u}_n}|x|^{2N}e^{\zeta}e^{w_n}[1-\frac{1}{2}e^{\bar{u}_n}|x|^{2N}e^{\zeta}e^{w_n}]dx+o(1)\notag\\&= - \int_{B_{s_n}(x_n)}2\lambda_ne^{\bar{u}_n}|x|^{2N}e^{\zeta}e^{w_n}[1-\frac{1}{2}e^{\bar{u}_n}|x|^{2N}e^{\zeta}e^{w_n}]dx+o(1),
\end{align} where the last equality follows from \eqref{equ:10}.

By lemma \ref{lem:01}, we know $$\nabla w_n(x)=-\frac{x-x_n}{|x-x_n|^2}\alpha_n(s_n;x_n)+\frac{o(1)}{|x|},\ \ x\in\partial B_{\sqrt{R_n}s_n}.$$

Then we conclude
\begin{align} \int_{B_{s_n}(x_n)}2\lambda_ne^{\bar{u}_n}|x|^{2N}e^{w_n}[1-\frac{1}{2}e^{\bar{u}_n}|x|^{2N}e^{w_n}]dx=\pi\alpha^2_n(s_n;x_n)+o(1),
\end{align} which yields the first conclusion of lemma.

\

\textbf{Step 2:} If $\frac{|x_n|}{s_n}\leq C_1$.

\

By Lemma \ref{lem:fast-decay}, we know there exists $$R_n\to\infty,\ \ R_ns_n=o(1)dist(\Sigma_n\setminus \Sigma^1_n,\partial B_{r_n}(x_n)),$$ such that $w_n$ has fast decay in $B_{2R_ns_n}(x_n)\setminus B_{s_n}(x_n)$, i.e. $$w_n(x)+2\log dist(x,\Sigma_n)+ 2N\log |x|+\log\lambda_ne^{\bar{u}_n}\leq -N_n, \ \forall s_n\leq |x-x_n|\leq 2R_ns_n,$$ for some $N_n\to\infty$ and
\begin{equation}\label{equ:11}
\int_{B_{2R_ns_n}(x_n)\setminus B_{s_n}(x_n)}\lambda_n e^{\bar{u}_n}|x|^{2N}e^{w_n(x)}dx=o(1).
\end{equation}

Multiplying \eqref{equ:02} by $x\nabla w_n$ and integrating by parts, we get the following Pohozaev type identity that
\begin{align}\label{equ:pohozaev-2}
\int_{\partial B_{s}(0)}r\left(\frac{1}{2}|\nabla w_n|^2-|\frac{\partial w_n}{\partial r}|^2\right)d\theta=& \int_{\partial B_{s}(0)}r\lambda_ne^{\bar{u}_n}|x|^{2N}e^{\zeta}e^{w_n}[1-\frac{1}{2}e^{\bar{u}_n}|x|^{2N}e^{\zeta}e^{w_n}]d\theta\notag\\ &- \int_{B_{s}(0)}2\lambda_ne^{\bar{u}_n}|x|^{2N}e^{\zeta}e^{w_n}[1-\frac{1}{2}e^{\bar{u}_n}|x|^{2N}e^{\zeta}e^{w_n}]dx\notag\\
& - \int_{B_{s}(0)}2N\lambda_ne^{\bar{u}_n}|x|^{2N}e^{\zeta}e^{w_n}[1-e^{\bar{u}_n}|x|^{2N}e^{\zeta}e^{w_n}]dx\notag\\
& - \int_{B_{s}(0)}\lambda_ne^{\bar{u}_n}|x|^{2N}( x\cdot \nabla \zeta )e^{\zeta}e^{w_n}[1-e^{\bar{u}_n}|x|^{2N}e^{\zeta}e^{w_n}]dx
\end{align} for any $0\leq s\leq 1$, where $(r,\theta)$ is the polar coordinate system centered at point $0$.

Take $s=\sqrt{R_n}s_n$ in \eqref{equ:pohozaev-2}. Noting that $$B_{s_n}(x_n)\subset B_{\sqrt{R_n}s_n}(0)\subset B_{2R_ns_n}(x_n)$$ and $$dist(x,\Sigma_n)=dist(x,\Sigma^1_n)\geq \sqrt{R_n}s_n-|x_n|-r_n\geq \frac{1}{2}\sqrt{R_n}s_n,\ \ \forall \ x\in \partial B_s(0),$$ when $n$ is big, since $w_n$ has fast decay in $B_{2R_ns_n}(x_n)\setminus B_{s_n}(x_n)$, we have
\begin{align*}
\int_{\partial B_{s}(0)}r\lambda_ne^{\bar{u}_n}|x|^{2N}e^{\zeta}e^{w_n}[1-\frac{1}{2}e^{\bar{u}_n}|x|^{2N}e^{\zeta}e^{w_n}]d\theta\leq Ce^{-N_n}\frac{s^2}{\min_{\partial B_s(0)}dist^2(x,\Sigma_n)}=o(1).
\end{align*}

By \eqref{equ:11}, we have $$\int_{B_{\sqrt{R_n}s_n}(0)}2\lambda_ne^{\bar{u}_n}|x|^{2N}e^{w_n}[1-\frac{1}{2}e^{\bar{u}_n}|x|^{2N}e^{w_n}]dx= \int_{B_{s_n}(x_n)}2\lambda_ne^{\bar{u}_n}|x|^{2N}e^{w_n}[1-\frac{1}{2}e^{\bar{u}_n}|x|^{2N}e^{w_n}]dx+o(1)$$ and $$\int_{B_{\sqrt{R_n}s_n}(0)}2N\lambda_ne^{\bar{u}_n}|x|^{2N}e^{w_n}[1-e^{\bar{u}_n}|x|^{2N}e^{w_n}]dx= \int_{B_{s_n}(x_n)}2N\lambda_ne^{\bar{u}_n}|x|^{2N}e^{w_n}[1-e^{\bar{u}_n}|x|^{2N}e^{w_n}]dx+o(1).$$

Since $|\nabla \zeta|\leq C$, then $$\left|\int_{B_{s}(0)}\lambda_ne^{\bar{u}_n}|x|^{2N}( x\cdot \nabla \zeta )e^{\zeta}e^{w_n}[1-e^{\bar{u}_n}|x|^{2N}e^{\zeta}e^{w_n}]dx\right|\leq C\sqrt{R_n}s_n=o(1).$$

Then we arrive at
\begin{align}
\int_{\partial B_{\sqrt{R_n}s_n}(0)}r\left(\frac{1}{2}|\nabla w_n|^2-|\frac{\partial w_n}{\partial r}|^2\right)d\theta&= - \int_{B_{s_n}(x_n)}2\lambda_ne^{\bar{u}_n}|x|^{2N}e^{\zeta}e^{w_n}[1-\frac{1}{2}e^{\bar{u}_n}|x|^{2N}e^{\zeta}e^{w_n}]dx\notag\\
&\quad - \int_{B_{s_n}(x_n)}2N\lambda_ne^{\bar{u}_n}|x|^{2N}e^{\zeta}e^{w_n}[1-e^{\bar{u}_n}|x|^{2N}e^{\zeta}e^{w_n}]dx.
\end{align}

By lemma \ref{lem:01}, we know $$\nabla w_n(x)=-\frac{x-x_n}{|x-x_n|^2}\alpha_n(s_n;x_n)+\frac{o(1)}{|x-x_n|}=-\frac{x}{|x|^2}\alpha_n(s_n;x_n)+\frac{o(1)}{|x|},\ \ x\in\partial B_{\sqrt{R_n}s_n}(0).$$ Then we conclude
\begin{align} \int_{B_{s_n}(x_n)}2\lambda_ne^{\bar{u}_n}|x|^{2N}e^{w_n}[1-\frac{1}{2}e^{\bar{u}_n}|x|^{2N}e^{w_n}]dx=\pi\alpha^2_n(s_n;x_n)-4\pi N\alpha_n(s_n;x_n)+o(1).
\end{align}

We finished the proof of this lemma.

 \end{proof}

\

Let $x^i_n\in \Sigma_n$. By Proposition \ref{prop:bubble-tree}, we know $$\varphi^i_n(x)=w_n(x^i_n+\mu_n^ix)-w_n(x^i_n)\to \varphi^i(x)$$ in the sense of $C^2_{loc}(\R^2)$, where $\varphi^i(x)$ satisfies \eqref{equ:07} or \eqref{equ:04}. Denote $$\int_{\R^2}|x-x_0^i|^{2N}e^{\varphi^i(x)}\left(1-\delta_i |x-x_0^i|^{2N}e^{\varphi(x)} \right)dx=2\pi b_i$$ or $$\int_{\R^2}e^{\varphi^i(x)}\left(1-\delta_i e^{\varphi^i(x)} \right)dx=2\pi b_i$$ corresponding to the fact that $(x_n^i,\mu_n^i)$ is a singular bubble or not.

\

At the end of this section, we prove following local energy identity.

\begin{prop}\label{prop:01}
Denote $$\tau^i_n:=\frac{1}{2}dist(x^i_n,\Sigma_n\setminus \{x^i_n\}).$$ Then the following two alternatives hold:
\begin{itemize}
 \item[(1)] if $(x^i_n,\mu_n^i)$ is a singular bubble, then $w_n(x)$ has fast decay on $\partial B_{\tau^i_n}(x^i_n)$ and $$\alpha_n(\tau_n^i;x^i_n)=b_i+o(1)$$ and $$\int_{B_{\tau^i_n}(x^i_n)}2\lambda_ne^{\bar{u}_n}|x|^{2N}e^{\zeta(x)}e^{w_n}[1-\frac{1}{2}e^{\bar{u}_n}|x|^{2N}e^{\zeta(x)}e^{w_n}]dx =\pi b_i^2-4\pi Nb_i+o(1).$$
 \item[(2)] if $(x^i_n,\mu_n^i)$ is not a singular bubble, denoting $s_n^i:=\min\{\tau_n^i,\frac{3}{4}|x_n^i|\}$, then $w_n(x)$ has fast decay on $\partial B_{s^i_n}(x^i_n)$ and $$\alpha_n(s_n^i;x^i_n)=b_i+o(1)$$ and $$\int_{B_{s^i_n}(x^i_n)}2\lambda_ne^{\bar{u}_n}|x|^{2N}e^{\zeta(x)}e^{w_n}[1-\frac{1}{2}e^{\bar{u}_n}|x|^{2N}e^{\zeta(x)}e^{w_n}]dx =\pi b_i^2+o(1).$$

\end{itemize}
\end{prop}

\begin{proof}
We divide the proof into two cases.

\

\textbf{Case 1:} If $(x^i_n,\mu_n^i)$ is a singular bubble.

\

In this case, by Proposition \ref{prop:bubble-tree}, we know $\frac{|x^i_n|}{\mu^i_n}\leq C$ and $$\varphi_n(x):=w_n(x_n^i+\mu_n^ix)-w_n(x_n^i)\to\varphi(x)$$ in the sense of $C^2_{loc}(\R^2)$, where $\varphi(x)$ satisfies \eqref{equ:04}.

Now, we claim:
\begin{equation}
\alpha_n(\tau^i_n;x_n^i)=b_i+o(1).
\end{equation}

In fact, by the proof of Proposition \ref{prop:bubble-tree}, we can choose $R_n\to\infty$ such that $R_n\mu_n^i=o(1)\tau^i_n$, $w_n(x)$ has fast decay on $\partial B_{\mu_n^i R_n}(x_n^i)$ and
\begin{equation}
\alpha_n(\mu_n^i R_n;x_n^i)=b_i+o(1),
\end{equation} where $b_i\geq 4(1+N)$.

Denote
\begin{equation}\label{equ:def-1}
\bar{w}_n(r):=\frac{1}{2\pi r}\int_{\partial B_r(x^i_n)}w_n(x) d\sigma,
\end{equation} where $(r,\theta)$ is the polar coordinates system centered at point $x^i_n$. Then we have
\begin{align}\label{ineq:10}
\frac{d}{dr}\bar{w}_n=\frac{1}{2\pi r}\int_{ B_r(x_n^i)}\Delta w_n(x) dx= -\frac{\alpha_n(r;x^i_n)}{r}&\leq -\frac{b_i+o(1)}{r}\notag\\ &\leq -\frac{4+4N+o(1)}{r},\ \ \forall \ \mu_n^iR_n\leq r\leq \tau^i_n,
\end{align} which implies $$\bar{w}_n(r)+2(1+N)\log r,\ \ \forall \ \mu_n^iR_n\leq r\leq \tau^i_n$$ and $$\bar{w}_n(r)+3+2N\log r,\ \ \forall \ \mu_n^iR_n\leq r\leq \tau^i_n$$ are decreasing functions with respect to $r$.

This together with the facts that $w_n(x)$ has fast decay on $\partial B_{\mu_n^iR_n}(x^i_n)$, $\frac{|x^i_n|}{\mu^i_n}\leq C$ and $$\frac{1}{2}|x-x_n^i|\leq |x|\leq 2|x-x_n^i|, \ \ \forall \ \mu_n^iR_n\leq |x-x_n^i|\leq \tau^i_n,$$  taking arbitrary point $p_n\in \partial B_{\mu_n^iR_n}(x_n^i)$, for any $\mu_n^iR_n\leq |x-x_n^i|\leq \tau^i_n$, by Lemma \ref{lem:osc}, we immediately get that
\begin{align*}
&w_n(x)+2\log r+2N\log |x|+2\log\lambda_ne^{\bar{u}_n}\\&\leq \bar{w}_n(r)+2\log r+2N\log r+2\log\lambda_ne^{\bar{u}_n}+C,\ \ (\mbox{here }r=|x-x_n^i|)\\
&\leq \bar{w}_n(\mu_n^iR_n)+2\log (\mu_n^iR_n) +2N\log (\mu_n^iR_n)+2\log\lambda_ne^{\bar{u}_n}+C\\
&\leq w_n(p_n)+2\log dist(p_n,\Sigma_n) +2N\log |p_n|+2\log\lambda_ne^{\bar{u}_n}+C\leq -N_n+C,
\end{align*}  which implies that $w_n(x)$ has fast decay in $B_{\tau_n^i}(x^i_n)\setminus  B_{\mu_n^iR_n}(x^i_n)$.

Moreover, we have
\begin{align*}
\alpha_n(\tau^i_n;x_n^i)-\alpha_n(\mu_n^i R_n;x_n^i)&\leq \int_{B_{\tau^i_n}(x_n^i)\setminus B_{\mu_n^iR_n}(x_n^i)}\lambda_ne^{\bar{u}_n}|x|^{2N}e^{w_n}dx\\ &\leq C\int_{B_{\tau^i_n}(x_n^i)\setminus B_{\mu_n^iR_n}(x_n^i)}\lambda_ne^{\bar{u}_n}|x|^{2N}e^{\bar{w}_n}dx\\ &\leq C\int_{B_{\tau^i_n}(x_n^i)\setminus B_{\mu_n^iR_n}(x_n^i)}\lambda_ne^{\bar{u}_n}|x-x_n^i|^{2N}e^{\bar{w}_n}dx\\ &\leq Ce^{\bar{w}_n(\mu_n^iR_n)+(3+2N)\log (\mu_n^iR_n)}\int_{B_{\tau^i_n}(x^i_n)\setminus B_{\mu_n^iR_n}(x^i_n)}\lambda_ne^{\bar{u}_n}|x-x_n^i|^{-3}dx \\ &\leq Ce^{\bar{w}_n(\mu_n^1R_n)+2(1+N)\log (\mu_n^1R_n)}\\ &\leq Ce^{w_n(p_n)+2\log dist(p_n,\Sigma_n) +2N\log |p_n|}=o(1),
\end{align*} where the second and last inequality follows from Lemma \ref{lem:osc}.

Thus, we get $$\alpha_n(\tau^i_n;x_n^i)=\alpha_n(\mu_n^i R_n;x_n^i)+o(1)=b_i+o(1).$$ Moreover, using global type Pohozaev identity of Lemma \ref{lem:pohozaev} in $B_{R_n\mu^i_n}(x^i_n)$, we get
\begin{align*}
&\int_{B_{\tau^i_n}(x^i_n)}2\lambda_ne^{\bar{u}_n}|x|^{2N}e^{\zeta(x)}e^{w_n}[1-\frac{1}{2}e^{\bar{u}_n}|x|^{2N}e^{\zeta(x)}e^{w_n}]dx \\&=\int_{B_{R_n\mu^i_n}(x^i_n)}2\lambda_ne^{\bar{u}_n}|x|^{2N}e^{\zeta(x)}e^{w_n}[1-\frac{1}{2}e^{\bar{u}_n}|x|^{2N}e^{\zeta(x)}e^{w_n}]dx =\pi b_i^2-4\pi Nb_i+o(1).
\end{align*}

\

\textbf{Case 2:} If $(x^i_n,\mu_n^i)$ is not a singular bubble.

\

In this case, we know $\frac{|x_n^i|}{\mu_n^i}\to \infty$. By Proposition \ref{prop:bubble-tree}, we know $$\varphi_n(x):=w_n(x_n^i+\mu_n^ix)-w_n(x_n^i)\to\varphi(x)$$ in the sense of $C^2_{loc}(\R^2)$, where $\varphi(x)$ satisfies \eqref{equ:07}. We choose $R_n\to\infty$ such that $R_n\mu_n^i=o(1)\tau_n^i$ and $R_n\mu_n^i=o(1)|x_n^i|$, $w_n(x)$ has fast decay on $\partial B_{\mu_n^i R_n}(x_n^i)$ and
\begin{equation}
\alpha_n(\mu_n^i R_n;x_n^i)=b_i+o(1),
\end{equation} where $b_i\geq 4$.

Denote $s_n^i:=\min\{\tau_n^i,\frac{3}{4}|x_n^i|\}$, by \eqref{ineq:10}, we have
\begin{align}
\frac{d}{dr}\bar{w}_n=\frac{1}{2\pi r}\int_{ B_r(x_n^i)}\Delta w_n(x) dx= -\frac{\alpha_n(r;x^i_n)}{r}&\leq -\frac{b^i+o(1)}{r}\notag\\ &\leq -\frac{4+o(1)}{r},\ \ \forall \ \mu_n^iR_n\leq r\leq s^i_n,
\end{align} which implies $$\bar{w}_n(r)+3\log r,\ \ \forall \ \mu_n^iR_n\leq r\leq s^i_n$$ is a decreasing function with respect to $r$. Noting that for any $\mu_n^iR_n\leq |x-x_n^i|\leq s^i_n$, there holds $$\frac{1}{4}|x_n^i|\leq |x|\leq 2|x_n^i|.$$ Taking any point $p_n\in \partial B_{\mu_n^iR_n}(x_n^i)$, we have $$\bar{w}_n(r)+3\log r+2N\log |x|\leq \bar{w}_n(\mu_n^iR_n)+3\log (\mu_n^iR_n) +2N\log |p_n|+C, \ \forall \ \mu_n^iR_n\leq r\leq s^i_n,$$ which implies that for any $\mu_n^iR_n\leq |x-x_n^i|\leq s^i_n$, there holds
\begin{align*}
&w_n(x)+2\log r+2N\log |x|+2\log\lambda_ne^{\bar{u}_n}\\ &\leq \bar{w}_n(r)+2\log r+2N\log|x|+2\log\lambda_ne^{\bar{u}_n}+C\\
&\leq \bar{w}_n(\mu_n^iR_n)+2\log (\mu_n^iR_n) +2N\log |p_n|+\log\frac{\mu_n^iR_n}{r}+2N\log\frac{|x|}{|p_n|}+2\log\lambda_ne^{\bar{u}_n}+C\\
&\leq \bar{w}_n(\mu_n^iR_n)+2\log (\mu_n^iR_n) +2N\log |p_n|+2\log\lambda_ne^{\bar{u}_n}+C\leq -N_n+C.
\end{align*}
Moreover, we have
\begin{align*}
\alpha_n(s^i_n;x_n^i)-\alpha_n(\mu_n^i R_n;x_n^i) &\leq C\int_{B_{s^i_n}(x_n^i)\setminus B_{\mu_n^iR_n}(x_n^i)}\lambda_ne^{\bar{u}_n}|x|^{2N}e^{\bar{w}_n}dx\\ &\leq C|x_n^i|^{2N}\lambda_ne^{\bar{u}_n} \int_{B_{s^i_n}(x_n^i)\setminus B_{\mu_n^iR_n}(x_n^i)}e^{\bar{w}_n}dx\\ &\leq C|x_n^i|^{2N}\lambda_ne^{\bar{u}_n} e^{\bar{w}_n(\mu_n^iR_n)+3\log (\mu_n^iR_n)}\int_{B_{s^i_n}(x^i_n)\setminus B_{\mu_n^iR_n}(x^i_n)}|x-x_n^i|^{-3}dx \\ &\leq C|x_n^i|^{2N}\lambda_ne^{\bar{u}_n} e^{\bar{w}_n(\mu_n^iR_n)+2\log (\mu_n^iR_n)}\\ &= C\left(\frac{|x_n^i|}{|p_n|}\right)^{2N}\lambda_ne^{\bar{u}_n} e^{\bar{w}_n(\mu_n^iR_n)+2\log (\mu_n^iR_n)+2N\log |p_n|}=o(1).
\end{align*}

Thus, we get $$\alpha_n(s^i_n;x_n^i)=\alpha_n(\mu_n^i R_n;x_n^i)+o(1)=b_i+o(1).$$ Moreover, using local type Pohozaev identity of Lemma \ref{lem:pohozaev} in $B_{R_n\mu^i_n}(x^i_n)$, we get
\begin{align*}
&\int_{B_{s^i_n}(x^i_n)}2\lambda_ne^{\bar{u}_n}|x|^{2N}e^{\zeta(x)}e^{w_n}[1-\frac{1}{2}e^{\bar{u}_n}|x|^{2N}e^{\zeta(x)}e^{w_n}]dx \\&=\int_{B_{R_n\mu^i_n}(x^i_n)}2\lambda_ne^{\bar{u}_n}|x|^{2N}e^{\zeta(x)}e^{w_n}[1-\frac{1}{2}e^{\bar{u}_n}|x|^{2N}e^{\zeta(x)}e^{w_n}]dx =\pi b_i^2+o(1).
\end{align*}

We proved this proposition.

\end{proof}

\

\section{Proof of theorems}\label{sec:proof}

\

In this section, we will prove Theorem \ref{thm:main-1}, Theorem \ref{thm:main-2} and Theorem \ref{thm:main} based on a induction argument on the numbers of bubbles. Here, we need some information of positional structure of bubbles which will be got by Pohozaev type identities. As a corollary of Theorem \ref{thm:main}, the proof of Theorem \ref{thm:main-3} will also be given in this section.

\

Before proving our main theorems and in order to describe the induction process better, we first give some definitions for positional structure of bubbling disks.

\begin{defn}
We call $B_{s_n}(x_n)$ a bubbling disk if $w_n(x)$ has fast decay on $\partial B_{s_n}(x_n)$ and $$\alpha_n(s_n;x_n)=\sum_{i=k_1}^{k_2}b_i+o(1),$$ where $b_i$ is the energy of the bubble $\varphi^i(x)$ in Proposition \ref{prop:bubble-tree}.
\end{defn}

\

\begin{defn}
Let $B_{s_n^i}(y^i_n),\ i=1,...,m$ be bubbling disks which are pairwise disjoint. Let $\{i_1,i_2,...,i_I\}\subset \{1,2,...,m\},\ I\geq 2,$ be a subset and $B_{r_n}(x_n)$ be a ball such that $$\cup_{i=i_1}^{i_I}B_{s_n^i}(y^i_n)\subset B_{r_n}(x_n),\ \ r_n=o(1)dist(\partial B_{r_n}(x_n), \{y_n^1,...,y_n^m\}\setminus \{y_n^{i_1},...,y_n^{i_I}\}).$$ We call $B_{s_n^{i_j}}(y^{i_j}_n), \ 1\leq j\leq I$, a locally smallest bubbling disk in $B_{r_n}(x_n)$ if there  exists a positive constant $C>1$ such that $$\limsup_n\frac{s_n^{i_k}}{s_n^{i_j}}\geq C^{-1},\ k=1,...,I.$$

\end{defn}

\begin{defn}
Let $B_{s_n^i}(y^i_n),\ i=1,...,m$ be  bubbling disks which are pairwise disjoint. Denote $S_n:=\{y_n^1,...,y_n^m\}$. For a point $x_n$, we call $B_{s_n^j}(y^j_n)$ a nearest bubbling disk corresponding to $x_n$, if $$dist(x_n,S_n\setminus \{x_n\})=|x_n-y_n^j|.$$

For a bubbling disk $B_{s^i_n}(y^i_n)$, we call $B_{s_n^j}(y^j_n)$ is a almost nearest bubbling disk corresponding to $B_{s^i_n}(y^i_n)$ if $$1\leq \frac{|y^j_n-y^i_n|}{dist(y_n^i,S_n\setminus \{y_n^i\})}\leq C.$$
\end{defn}

\

\begin{rem}\label{rem:01}
By above definitions, we immediately have the following facts that
\begin{itemize}
\item[(1)] For $x_n^i\in\Sigma_n$, by Proposition \ref{prop:01}, we know $B_{s_n^i}(x_n^i)$ is a bubbling disk, where $s_n^i:=\min\{\tau_n^i,\frac{3}{4}|x_n^i|\}$ and $\tau_n^i:=\frac{1}{2}dist(x_n^i,\Sigma_n\setminus \{x_n^i\})$.

\

\item[(2)] Let $B_{\tau_n^i}(y^i_n),\ i=1,...,m$ be bubbling disks where $\tau_n^i:=\frac{1}{2}dist(y^i_n,S_n\setminus \{y_n^i\})$ and $S_n:=\{y_n^1,...,y_n^m\}$. For $B_{\tau_n^i}(y^i_n)$, there exist some almost nearest bubbling disks (at least one which is the nearest one), denoted by $\{B_{\tau_n^{i_j}}(y^{i_j}_n)\}_{j=1}^J,\ i_1,...,i_j\in \{1,...,i-1,i+1,...,m\}$. Then there exists a positive constant $C$ such that $$\bigcup_{j=1}^JB_{\tau_n^{i_j}}(y^{i_j}_n)\subset B_{C\tau_n^i}(y^i_n),\ \ \tau_n^i=o(1)dist(\partial B_{C\tau_n^i}(y^i_n),S_n\setminus \{y_n^i,y_n^{i_1},...,y_n^{i_J}\}).$$

\

\item[(3)]If $B_{\tau_n^i}(y^i_n)$ is a locally smallest bubbling disk in $B_{C\tau_n^i}(y^i_n)$ and $B_{\tau_n^j}(y^j_n)\ (j\in\{i_1,...,i_J\})$ is a almost nearest bubbling disk corresponding to $B_{\tau^i_n}(y^i_n)$, then we have $$C^{-1}\tau_n^i\leq \tau_n^j=\frac{1}{2}dist(y^j_n,S_n\setminus \{y_n^j\})\leq \frac{1}{2}|y_n^j-y_n^i|\leq C\tau_n^i.$$
\end{itemize}
\end{rem}

\

Next, we prove two propositions to simplify the structure of bubbling disks.
\begin{prop}\label{prop:02}
Let $x_n^1\in\Sigma_n$ and $(x_n^1,\mu_n^1)$ be not a singular bubble.  Then either $\Sigma_n=\{x_n^1\}$ or $\lim\frac{|x_n^1|}{\tau_n^1}\leq C$.
\end{prop}
\begin{proof}
Suppose $\Sigma_n$ contains more than two points, we will show $\lim\frac{|x_n^1|}{\tau_n^1}\leq C$. We prove by a contradiction argument. If not, then passing to a subsequence, there holds $\lim\frac{|x_n^1|}{\tau_n^1}=\infty$. Then by Proposition \ref{prop:01}, we know that $B_{\tau_n^1}(x^1_n)$ is a bubbling disk, i.e. $s_n^1=\tau_n^1$.

Since, there are more than two bubbles, for $B_{\tau_n^1}(x_n^1)$, there must be a nearest bubbling disk denoted by $B_{\tau_n^2}(x_n^2)$ and maybe some almost nearest bubbling disks denoted by $B_{\tau_n^i}(x_n^i),\ i=3,...,I$. It is easy to see that $$\frac{1}{2}|x_n^2-x_n^1|=\tau_n^1=o(1)|x_n^1|,\ \ |x_n^i-x_n^1|\leq C|x_n^2-x_n^1|\leq C\tau_n^1=o(1)|x_n^1|,\ i=2,...,I,$$ and  $$\cup_{i=1}^{I}B_{\tau_n^i}(x_n^i)\subset B_{\Lambda\tau_n^1}(x_n^1)$$ for some positive constant $\Lambda>0$. Moreover, since $$\tau_n^i\leq 2|x_n^i-x_n^1|=o(1)|x_n^1|=o(1)|x_n^i|,\ \ i=2,...,I,$$ we know $\{(x_n^i,\mu_n^i)\}_{ i=1}^I$ are not singular bubbles and $$\tau_n^1=o(1)dist(\partial B_{\Lambda\tau_n^1}(x_n^1),\Sigma_n\setminus \{x_n^1,...,x_n^I\})$$ according to the definition of almost nearest bubbling disk.

\

Since $2\tau_n^i\leq |x_n^i-x_n^1|\leq C\tau_n^1,\ \ i=2,...,I$, we shall consider following two cases:

\

\textbf{Case 1:} if $\lim\frac{\tau_n^i}{\tau_n^1}\geq C^{-1},\ i=2,...,I$, for some positive constant $C>0$.

\

In this case, we have $C^{-1}\leq\frac{\tau_n^1}{\tau_n^i}\leq C,\ i=2,...,I$. Now, we claim: $w_n(x)$ has fast decay on $\partial B_{\Lambda\tau_n^1}(x_n^1)$ and
\begin{equation}\label{equ:14}
\alpha_n(\Lambda_n\tau_n^1;x_n^1)=\sum_{i=1}^{I}\alpha_n(\tau_n^i;x_n^i)+o(1)=\sum_{i=1}^{I}b_i+o(1).
\end{equation}

For \eqref{equ:14}, we  just need to show the first equality because the second one follows from Proposition \ref{prop:01}.

In fact, denoting $\Omega_n:=B_{\Lambda\tau_n^1}(x_n^1)\setminus \left(\cup_{i=1}^{I}B_{\tau_n^i}(x_n^i)\right)$, by Lemma \ref{lem:osc}, we firstly have $$osc_{\Omega_n}w_n(x)\leq C.$$ Secondly, for any $x\in\Omega_n$, it is easy to see that $$\frac{1}{2}|x_n^1|\leq |x_n^1|-|x-x_n^1|\leq |x|\leq |x-x_n^1|+|x_n^1|\leq 2|x_n^1|$$ and $$C^{-1}\tau_n^1\leq \tau_n^i\leq |x-x_n^i|\leq |x-x_n^1|+|x_n^1-x_n^i|\leq C\tau_n^1,\ i=1,...,I,$$ which implies $$osc_{\Omega_n}\log dist(x,\{x^1_n,...,x_n^I\})+osc_{\Omega_n}\log|x|\leq C.$$ Since $w_n(x)$ has fast decay on $\partial B_{\tau_n^i}(x_n^i)$, we get that $w_n(x)$ has fast decay in $\Omega_n$ and
\begin{align}\label{ineq:13}
\int_{\Omega_n}\lambda_ne^{\bar{u}_n}|x|^{2N} e^{\zeta}e^{w_n}dx&\leq o(1)\int_{\Omega_n}\frac{1}{|dist(x,\{x_n^1,x_n^2,...,x_n^{I}\})|^2}dx\notag\\
&\leq  o(1)\sum_{i=1}^{I}\int_{\Omega_n}\frac{1}{|x-x_n^i|^2}dx=o(1),
\end{align} which implies \eqref{equ:14} immediately.

Now we estimate the integration $$2\int_{B_{\Lambda\tau_n^1}(x_n^1)}\lambda_ne^{\bar{u}_n}|x|^{2N}e^{\zeta} e^{w_n}(1-\frac{1}{2}e^{\bar{u}_n}|x|^{2N}e^{\zeta}  e^{w_n})dx.$$

On one hand, by using \eqref{ineq:13} and local type Pohozaev identity of Lemma \ref{lem:pohozaev} in $B_{\tau_n^i}(x_n^i),\ i=1,...,I$, we obtain
\begin{align*}
&2\int_{B_{\Lambda\tau_n^1}(x_n^1)}\lambda_ne^{\bar{u}_n}|x|^{2N} e^{\zeta}e^{w_n}(1-\frac{1}{2}e^{\bar{u}_n}|x|^{2N}e^{\zeta} e^{w_n})dx\\&=2\sum_{i=1}^I\int_{B_{\tau_n^i}(x_n^i)}\lambda_ne^{\bar{u}_n}|x|^{2N}e^{\zeta} e^{w_n}(1-\frac{1}{2}e^{\bar{u}_n}|x|^{2N}e^{\zeta} e^{w_n})dx=\pi\sum_{i=1}^{I}\alpha_n^2(x_n^i;\tau_n^i)+o(1)=\pi\sum_{i=1}^{I}b_i^2+o(1).
\end{align*}

On the other hand, using local type Pohozaev identity of Lemma \ref{lem:pohozaev} in $B_{\Lambda\tau_n^1}(x_n^1)$, we have
\begin{equation*}
2\int_{B_{\Lambda\tau_n^1}(x_n^1)}\lambda_ne^{\bar{u}_n}|x|^{2N} e^{w_n}(1-\frac{1}{2}e^{\bar{u}_n}|x|^{2N} e^{w_n})dx=\pi\alpha_n^2(x_n^1;\Lambda\tau_n^1)+o(1)=\pi\left(\sum_{i=1}^{I}b_i\right)^2+o(1).
\end{equation*}
Then we get $$\sum_{i=1}^{I}b_i^2=\left(\sum_{i=1}^{I}b_i\right)^2,$$ which  is a contradiction because we have at least two bubbles, i.e. $I\geq 2$.

\

\textbf{Case 2:} There exists $i\in\{2,...,I\}$ such that $\lim\frac{\tau_n^i}{\tau_n^1}=0$.

\

In this case, there exists a locally smallest bubbling disk (denoted by $B_{\tau_n^{2}}(x_n^2)$) in $B_{\Lambda\tau_n^1}(x_n^1)$, such that $$\frac{\tau_n^1}{\tau_n^{2}}\to\infty,\ \ \frac{\tau_n^j}{\tau_n^{2}}\geq C^{-1}>0,\ j\in\{2,...,I\}.$$ Then we know there must be more than three bubbles, i.e. $I\geq 3$. Otherwise, $\tau_n^2=\tau_n^1$ which is a contradiction.

Furthermore, for $B_{\tau_n^{2}}(x_n^{2})$, there must exist a nearest bubble (denoted by $B_{\tau_n^{3}}(x_n^{3})$) and maybe some almost nearest bubbles denoted by $B_{\tau_n^{4}}(x_n^{4}),...,B_{\tau_n^{I}}(x_n^{I})$. By Remark \ref{rem:01}, we get $$|x_n^{3}-x_n^{2}|=2\tau_n^{2},\ \ 1\leq \frac{|x_n^{j}-x_n^{2}|}{|x_n^{3}-x_n^{2}|}\leq C,\ \ C^{-1}\tau_n^{2}\leq\tau_n^j\leq \frac{1}{2} |x_n^j-x_n^{2}|  \leq C\tau_n^{2},\  j=3,...,I$$ and there exists a constant $\Lambda_1>100$ such that
\begin{align*}
\cup_{j=2}^{I}B_{\tau_n^j}(x_n^j)\subset B_{\Lambda_1\tau_n^{2}}(x_n^{2}),\ \ \tau_n^2=o(1)dist(\partial B_{\Lambda_1\tau_n^{2}}(x_n^{2}),\Sigma_n\setminus \Sigma_n^1),
\end{align*} where $\Sigma_n^1:=\{x_n^{2},...,x_n^{I}\}$.

Similarly to $case \ 1$, we can prove that $w_n(x)$ has fast decay on $\partial B_{\Lambda_1\tau_n^{2}}(x_n^{2})$ and $$\alpha_n(\Lambda_1\tau_n^{2};x_n^{2})=\sum_{j=2}^{I}b_j+o(1).$$ Moreover, by using local Pohozaev identity in each bubbling disk $B_{\tau_n^j}(x_n^j),\ j=2,...,I$ and $ B_{\Lambda_1\tau_n^{2}}(x_n^{2})$, we obtain $$\sum_{j=2}^{I}b_j^2=\left(\sum_{j=2}^{I}b_j\right)^2+o(1),$$ which is also a contradiction because $\Sigma_n^1$ contains at least two points, i.e. $I\geq 3$.

We proved proposition.
\end{proof}

\

\begin{prop}\label{prop:03}
Let $x_n^1\in\Sigma_n$ and $(x_n^1,\mu_n^1)$ be a singular bubble.  Then $\Sigma_n=\{x_n^1\}$, i.e. there is only one bubble.
\end{prop}
\begin{proof}
We prove by a contradiction argument. If not, $\Sigma_n=\{x_n^1,...,x_n^m\}$ have more than two points.    In this case, by Proposition \ref{prop:bubble-tree}, we know $|x_n^1|=o(1)|x_n^i|,\ i=2,...,m$. Then $$\tau_n^i\leq \frac{1}{2}|x_n^i-x_n^1|\leq \frac{3}{4}|x_n^i|,\ i=2,...,m.$$ By Proposition \ref{prop:01}, we know $B_{\tau_n^i}(x_n^i),\ i=1,...,m$ are all bubbling disks, i.e. $s_n^i=\tau_n^i$.

For bubbling disk $B_{\tau_n^1}(x^1_n)$, we can find a nearest bubble (denoted by $(x_n^2,\tau_n^2)$) and maybe almost nearest bubbles (denoted by $(x_n^i,\tau_n^i),\ i=3,...,I$). By definition, we know  $$1 \leq \frac{|x^i_n-x^1_n|}{|x^2_n-x^1_n|}\leq C,\ \ \tau_n^i\leq \frac{1}{2}|x_n^i-x_n^1|\leq C \tau_n^1,\ i=2,...,I$$ and   $$\cup_{i=1}^{I}B_{\tau_n^i}(x_n^i)\subset B_{\Lambda\tau_n^1}(x_n^1)$$ for some positive constant $\Lambda>0$. Moreover, according to the definition of almost nearest bubbling disk, we know  $$\tau_n^1=o(1)dist(\partial B_{\Lambda\tau_n^1}(x_n^1),\Sigma_n\setminus \{x_n^1,...,x_n^I\}).$$

\

Since $ \lim\frac{\tau_n^i}{\tau_n^1}\leq C,\ i=2,...,I$, we shall consider the following cases:

\

\textbf{Case I:} if $\lim\frac{\tau_n^i}{\tau_n^1}\geq C^{-1},\ i=2,...,I$, for some positive constant $C$.

\

In this case, we have $C^{-1}\leq\frac{\tau_n^i}{\tau_n^1}\leq C,\ i=2,...,I$. Now, we claim: $w_n(x)$ has fast decay on $\partial B_{\Lambda\tau_n^1}(x_n^1)$ and
\begin{equation}\label{equ:13}
\alpha_n(\Lambda_n\tau_n^1;x_n^1)=\sum_{i=1}^{I}\alpha_n(\tau_n^i;x_n^i)+o(1)=\sum_{i=1}^{I}b_i+o(1).
\end{equation}
In fact, denoting $\Omega_n:=B_{\Lambda\tau_n^1}(x_n^1)\setminus \left(\cup_{i=1}^{I}B_{\tau_n^i}(x_n^i)\right)$, by Lemma \ref{lem:osc}, we firstly have $$osc_{\Omega_n}w_n(x)\leq C.$$ Secondly, for any $x\in\Omega_n$, it is easy to see that $$\frac{1}{2}\tau_n^1\leq |x-x_n^1|-|x_n^1|\leq |x|\leq |x-x_n^1|+|x_n^1|\leq (1+\Lambda)\tau_n^1$$ and $$C^{-1}\tau_n^1\leq \tau_n^i\leq |x-x_n^i|\leq |x-x_n^1|+|x_n^1-x_n^i|\leq C\tau_n^1,\ i=1,...,I,$$ which implies $$osc_{\Omega_n}\log dist(x,\{x^1_n,...,x_n^I\})+osc_{\Omega_n}\log|x|\leq C.$$ Since $w_n(x)$ has fast decay on $\partial B_{\tau_n^i}(x_n^i)$, we get that $w_n(x)$ has fast decay in $\Omega_n$ and
\begin{align}\label{ineq:12}
\int_{\Omega_n}\lambda_ne^{\bar{u}_n}|x|^{2N} e^{w_n}dx&\leq o(1)\int_{\Omega_n}\frac{1}{|dist(x,\{x_n^1,x_n^2,...,x_n^{I}\})|^2}dx\notag\\
&\leq  o(1)\sum_{i=1}^{I}\int_{\Omega_n}\frac{1}{|x-x_n^i|^2}dx=o(1),
\end{align} which implies \eqref{equ:13} immediately.

Similarly to Proposition \ref{prop:02}, on one hand, by using \eqref{ineq:12} and Proposition \ref{prop:01}, we obtain
\begin{align*}
&2\int_{B_{\Lambda\tau_n^1}(x_n^1)}\lambda_ne^{\bar{u}_n}|x|^{2N} e^{w_n}(1-\frac{1}{2}e^{\bar{u}_n}|x|^{2N} e^{w_n})dx\\&=2\sum_{i=1}^I\int_{B_{\Lambda\tau_n^i}(x_n^i)}\lambda_ne^{\bar{u}_n}|x|^{2N} e^{w_n}(1-\frac{1}{2}e^{\bar{u}_n}|x|^{2N} e^{w_n})dx=\pi\sum_{i=1}^{I}b_i^2-4\pi Nb_1+o(1).
\end{align*}

On the other hand, using global type Pohozaev identity of Lemma \ref{lem:pohozaev} in $B_{\Lambda\tau_n^1}(x_n^1)$, we have
\begin{equation}
2\int_{B_{\Lambda\tau_n^1}(x_n^1)}\lambda_ne^{\bar{u}_n}|x|^{2N} e^{w_n}(1-\frac{1}{2}e^{\bar{u}_n}|x|^{2N} e^{w_n})dx=\pi\left(\sum_{i=1}^{I}b_i\right)^2-4\pi N\sum_{i=1}^{I}b_i+o(1).
\end{equation}
Then we get $$\sum_{i=1}^{I}b_i^2-4 Nb_1=\left(\sum_{i=1}^{I}b_i\right)^2-4 N\sum_{i=1}^{I}b_i,$$ which implies $$\sum_{2\leq i,j\leq I,i\neq j}b_ib_j=(2N-b_1)\sum_{i=2}^{I}b_i.$$ This is a contradiction since $b_1\geq 4N+4$ and $2N-b_1<0$.

\

\textbf{Case  II:}  if there exists $i\in \{2,...,I\}$ such that $\lim\frac{\tau_n^i}{\tau_n^1}=0$.

\

By a almost same argument as in case 2 of Proposition \ref{prop:02}, we will also get a contradiction.

We proved the proposition.

\end{proof}

\

Now, at the end of this section, we prove our main theorems.

\

\begin{proof}[\textbf{Proof of  Theorem \ref{thm:main-1}.}] Firstly, by Lemma \ref{lem:02}, we get $c=0$. Next, We divide the proof into the following two steps.

\

\textbf{Step 1:} If $\Sigma_n$ contains one point $\{x_n^1\}$ and $(x_n^1,\mu_n^1)$ is a singular bubble, then  conclusion $(2-1)$ follows.

\

By Proposition \ref{prop:03}, we know $\Sigma_n=\{x_n^1\}$. Since $(x_n^1,\mu_n^1)$ is the only singular bubble and $\tau_n^1\geq \frac{1}{2}$, by Proposition \ref{prop:01}, we have $$\alpha_n(\frac{1}{2};x^1_n)=b_1+o(1),$$ which implies $$\lim_{\rho\to 0}\lim_{n\to\infty}\alpha_n(\rho;0)=b_1.$$ Thus, conclusion $(2-1)$ follows immediately.

\

\textbf{Step 2:} If $\Sigma_n$ contains a point $x_n^i$ such that $(x_n^i,\mu_n^i)$ is not a singular bubble, then conclusion $(2-2)$ follows.

\

For simplicity of notation, we assume $i=1$. In this case, by proposition \ref{prop:02}, we know either $\Sigma_n=\{x_n^1\}$ or $\lim\frac{|x_n^1|}{\tau_n^1}\leq C$.

We first claim that $\Sigma_n$ contains more than two points. In fact, if not, since $\Sigma_n$ contains only one point $\{x_n^1\}$, then  we have $\tau_n^1\geq \frac{1}{2}$  which implies $s_n^1=\frac{3}{4}|x_n^1|$.

On one hand, since $(x_n^1,\mu_n^1)$ is not a singular bubble, by Proposition \ref{prop:01}, we have $w_n(x)$ has fast decay on $\partial B_{s_n^1}(x_n^1)$ and $$\alpha_n(s_n^1;x^1_n)=b_1+o(1)$$ and $$\int_{B_{s^1_n}(x^1_n)}2\lambda_ne^{\bar{u}_n}|x|^{2N}e^{\zeta(x)}e^{w_n}[1-\frac{1}{2}e^{\bar{u}_n}|x|^{2N}e^{\zeta(x)}e^{w_n}]dx =\pi b_1^2+o(1),$$ where $b_1\geq 4$.

On the other hand, since $\lim\frac{|x_n^1|}{s_n^1}\leq C$, by global type Pohozaev identity of Lemma \ref{lem:pohozaev} in $B_{s^1_n}(x^1_n)$, we get $$\int_{B_{s^1_n}(x^1_n)}2\lambda_ne^{\bar{u}_n}|x|^{2N}e^{\zeta(x)}e^{w_n}[1-\frac{1}{2}e^{\bar{u}_n}|x|^{2N}e^{\zeta(x)}e^{w_n}]dx =\pi b_1^2-4\pi Nb_1+o(1).$$ This yields $$\pi b_1^2-4\pi Nb_1=\pi b_1^2+o(1),$$ which is a contradiction. Thus, $\Sigma_n$ must contain more than two points.

\

By Proposition \ref{prop:02} and \ref{prop:03}, we know that $\Sigma_n$ contains no singular bubble and $$\lim\frac{|x_n^i|}{\tau_n^i}\leq C,\ \ i=1,...,m.$$ By Proposition \ref{prop:01}, we know that $\{B_{s_n^i}(x_n^i)\}_{ i=1}^m$ are pairwise disjoint bubbling disks where $s_n^i:=\min\{\tau_n^i,\frac{3}{4}|x_n^i|\}$.  Then we shall consider the following cases.

\

\textbf{Case 1:} If there exists $i\in \{1,...,m\} $ such that $\lim\frac{|x_n^i|}{\tau_n^i}=0$.

\

For simplicity of notation, we assume $i=1$. In this case, we see that $s_n^1=\frac{3}{4}|x_n^1|$. On one hand, by Proposition \ref{prop:01}, we have
\begin{equation}
2\int_{B_{s_n^1}(x_n^1)}\lambda_ne^{\bar{u}_n}|x|^{2N}e^{\zeta} e^{w_n}(1-\frac{1}{2}e^{\bar{u}_n}|x|^{2N}e^{\zeta} e^{w_n})dx=\pi b_1^2+o(1).
\end{equation}
On the other hand, since $|x_n^1|=o(1)\tau_n^1$, then $s_n^1=o(1)dist(\partial B_{s_n^1}(x_n^1),\Sigma_n\setminus \{x_n^1\})$. Using global type Pohozaev identity of Lemma \ref{lem:pohozaev} in $B_{s_n^1}(x_n^1)$, we have
\begin{equation}
2\int_{B_{s_n^1}(x_n^1)}\lambda_ne^{\bar{u}_n}|x|^{2N}e^{\zeta} e^{w_n}(1-\frac{1}{2}e^{\bar{u}_n}|x|^{2N} e^{\zeta}e^{w_n})dx=\pi b_1^2-4\pi Nb_1+o(1).
\end{equation}
Thus, we get $$ b_1^2=b_1^2-4 Nb_1+o(1),$$ which is a contradiction.

\

\textbf{Case 2:} $\lim\frac{|x_n^i|}{\tau_n^i}\geq C^{-1},\ \ i=1,...,m$ for some constant $C>0$.

\

In this case, we first take a smallest bubbling disk denoted by $B_{s_n^1}(x_n^1)$. Next, we claim: $w_n(x)$ has fast decay on $\partial B_{\frac{1}{2}}(x_n^1)$ and
\begin{equation}\label{equ:19}
\alpha_n(\frac{1}{2};x_n^1)=\sum_{i=1}^{m}\alpha_n(s_n^i;x_n^i)+o(1)=\sum_{i=1}^{m}b_i+o(1)
\end{equation} and
\begin{align}\label{equ:20}
&2\int_{B_{\frac{1}{2}}(x_n^1)}\lambda_ne^{\bar{u}_n}|x|^{2N}e^{\zeta} e^{w_n}(1-\frac{1}{2}e^{\bar{u}_n}|x|^{2N}e^{\zeta} e^{w_n})dx\notag\\&=\sum_{i=1}^m2\int_{B_{ s_n^i}(x_n^i)}\lambda_ne^{\bar{u}_n}|x|^{2N}e^{\zeta} e^{w_n}(1-\frac{1}{2}e^{\bar{u}_n}|x|^{2N}e^{\zeta} e^{w_n})dx.
\end{align}

 We prove by a induction argument on the numbers of bubbling disks $B_{s_n^i}(x_n^i),\ i=1,...,m$.

\

\text{Step 2-1:} We prove for two bubbling disks, i.e. $m=2$.

\

In this case, we have two bubbling disks $B_{s_n^1}(x_n^1)$ and $B_{s_n^2}(x_n^2)$, where $B_{s_n^1}(x_n^1)$ is a smallest bubbling disk satisfying $$C^{-1}\tau_n^1\leq s_n^1\leq \tau_n^1 ,\ \ C^{-1}s_n^1\leq |x_n^1|\leq Cs_n^1.$$ It is easy to see that $$\tau_n^1=\tau_n^2,\ \ C^{-1}s_n^1\leq s_n^2\leq Cs_n^1,\ \  C^{-1}s_n^1\leq |x_n^2|\leq Cs_n^1.$$ By the proof of Proposition \ref{prop:02} or Proposition \ref{prop:03}, we can show that there exists a positive constant $\Lambda>100$ such that $B_{s_n^1}(x_n^1)\cup B_{s_n^2}(x_n^2)\subset B_{\Lambda s_n^1}(x_n^1)$, $w_n(x)$ has fast decay in $ B_{\Lambda s_n^1}(x_n^1)\setminus \cup_{i=1}^2B_{s_n^i}(x_n^i)$ and
\begin{equation*}
\alpha_n(\Lambda s_n^1;x_n^1)=\sum_{i=1}^{2}\alpha_n(s_n^i;x_n^i)+o(1)=\sum_{i=1}^{2}b_i+o(1)
\end{equation*} and
\begin{align*}
&2\int_{B_{\Lambda s_n^1}(x_n^1)}\lambda_ne^{\bar{u}_n}|x|^{2N}e^{\zeta} e^{w_n}(1-\frac{1}{2}e^{\bar{u}_n}|x|^{2N}e^{\zeta} e^{w_n})dx\notag\\&=\sum_{i=1}^22\int_{B_{ s_n^i}(x_n^i)}\lambda_ne^{\bar{u}_n}|x|^{2N}e^{\zeta} e^{w_n}(1-\frac{1}{2}e^{\bar{u}_n}|x|^{2N}e^{\zeta} e^{w_n})dx.
\end{align*}

Moreover, using a local and global type Pohozaev identity, we have the following relation $$b_1^2+b_2^2=(b_1+b_2)^2-4N(b_1+b_2),$$ which implies that $$b_1+b_2\geq 4(1+N).$$

By \eqref{ineq:10}, we have
\begin{align*}
\frac{d}{dr}\bar{w}_n=\frac{1}{2\pi r}\int_{ B_r(x_n^1)}\Delta w_n(x) dx&= -\frac{\alpha_n(r;x^1_n)}{r}\\&\leq -\frac{4(1+N)+o(1)}{r},\ \ \forall \ \Lambda s_n^1\leq r=|x-x_n^1|\leq \frac{1}{2}.
\end{align*}

Similarly to the proof of Proposition \ref{prop:01}, we can prove that $w_n(x)$ has fast decay in $B_{\frac{1}{2}}(x_n^1)\setminus B_{\Lambda s_n^1}(x_n^1)$ and $$\int_{B_{\frac{1}{2}}(x_n^1)\setminus \cup_{i=1}^2B_{s_n^i}(x_n^i)}\lambda_ne^{\bar{u}_n}|x|^{2N}e^{\zeta}e^{w_n}dx=o(1)$$ which immediately implies \eqref{equ:19} and \eqref{equ:20}. We proved for $m=2$.

\

\text{Step 2-2:} Assuming when the number of bubbling disks is less than $m-1$, we can show that $w_n(x)$ has fast decay on $\partial B_{\frac{1}{2}}(x_n^1)$ and \eqref{equ:19}, \eqref{equ:20} hold. Next, we prove it  still holds for $m$ bubbling disks.

\

In this case, there holds $C^{-1}\tau_n^1\leq |x_n^1|\leq C\tau_n^1$ which implies $C^{-1}\tau_n^1\leq s_n^1\leq C\tau_n^1$. Recall that $B_{s_n^1}(x_n^1)$ is a smallest bubbling disk. Since $\Sigma_n$ has more than two points, for $B_{s_n^1}(x_n^1)$, there must be a nearest bubbling disk denoted by $B_{s_n^2}(x_n^2)$ and maybe some almost nearest bubbling disks denoted by $B_{s_n^i}(x_n^i),\ i=3,...,I$. Noting that $$C^{-1}s_n^1\leq s_n^i\leq \tau_n^i\leq |x_n^i-x_n^1|\leq C|x_n^2-x_n^1|\leq C\tau_n^1\leq Cs_n^1,\ i=2,...,I,$$ then there exists a positive constant $\Lambda>0$ such that $$\cup_{i=1}^IB_{s_n^i}(x_n^i)\subset B_{\Lambda s_n^1}(x_n^1),\ \ and\ \ 0\in B_{\Lambda s_n^1}(x_n^1).$$

Noting that $s_n^1\leq s_n^i\leq Cs_n^1,\ i=2,...,I$, similar to the case 1 of Proposition \ref{prop:02} or case I of Proposition \ref{prop:03}, we can show that $$\alpha_n(\Lambda s_n^1;x_n^1)=\sum_{i=1}^Ib_i+o(1)$$
and
\begin{align*}
&2\int_{B_{\Lambda s_n^1}(x_n^1)}\lambda_ne^{\bar{u}_n}|x|^{2N}e^{\zeta} e^{w_n}(1-\frac{1}{2}e^{\bar{u}_n}|x|^{2N}e^{\zeta} e^{w_n})dx\notag\\&=\sum_{i=1}^I2\int_{B_{ s_n^i}(x_n^i)}\lambda_ne^{\bar{u}_n}|x|^{2N}e^{\zeta} e^{w_n}(1-\frac{1}{2}e^{\bar{u}_n}|x|^{2N}e^{\zeta} e^{w_n})dx=\sum_{i=1}^Ib_i^2+o(1).
\end{align*}

Writing $B_{\Lambda s_n^1}(x_n^1)$ as a new bubbling disk, together with the left bubbling disks $B_{s_n^i}(x_n^i),\ i=I+1,...,m$ and the induction assumption, we get  $w_n(x)$ has fast decay on $\partial B_{\frac{1}{2}}(x_n^1)$ and  $$\alpha_n(\frac{1}{2};x_n^1)=\alpha_n(\Lambda s_n^1;x_n^1)+\sum_{i=I}^m \alpha_n(\Lambda s_n^i;x_n^i)+o(1)=\sum_{i=1}^mb_i+o(1)$$ and
\begin{align*}
&2\int_{B_{\frac{1}{2}}(x_n^1)}\lambda_ne^{\bar{u}_n}|x|^{2N}e^{\zeta} e^{w_n}(1-\frac{1}{2}e^{\bar{u}_n}|x|^{2N}e^{\zeta} e^{w_n})dx\notag\\&=2\int_{B_{\Lambda s_n^1}(x_n^1)}\lambda_ne^{\bar{u}_n}|x|^{2N}e^{\zeta} e^{w_n}(1-\frac{1}{2}e^{\bar{u}_n}|x|^{2N}e^{\zeta} e^{w_n})dx\\&\quad + 2\sum_{i=I+1}^{m}\int_{B_{ s_n^i}(x_n^i)}\lambda_ne^{\bar{u}_n}|x|^{2N}e^{\zeta} e^{w_n}(1-\frac{1}{2}e^{\bar{u}_n}|x|^{2N}e^{\zeta} e^{w_n})dx.
\end{align*}

We finished the proof of induction, i.e. $w_n(x)$ has fast decay on $\partial B_{\frac{1}{2}}(x_n^1)$  and \eqref{equ:19}, \eqref{equ:20} hold.

We continue to show the conclusion of $(2-2)$. In fact, we first easy to see that \eqref{equ:19} implies $$\lim_{\rho\to 0}\lim_{n\to\infty}\alpha_n(\rho;0)=\sum_{i=1}^mb_i+o(1).$$  Secondly, on one hand, by \eqref{equ:20} and Proposition \ref{prop:01}, we have
\begin{align*}
&2\int_{B_{\frac{1}{2}}(x_n^1)}\lambda_ne^{\bar{u}_n}|x|^{2N}e^{\zeta} e^{w_n}(1-\frac{1}{2}e^{\bar{u}_n}|x|^{2N}e^{\zeta} e^{w_n})dx=\pi\sum_{i=1}^mb_i^2+o(1).
\end{align*}

On the other hand, since $w_n(x)$ has fast decay on $\partial B_{\frac{1}{2}}(x_n^1)$, by global type Pohozaev identity, we obtain
\begin{align*}
&2\int_{B_{\frac{1}{2}}(x_n^1)}\lambda_ne^{\bar{u}_n}|x|^{2N}e^{\zeta} e^{w_n}(1-\frac{1}{2}e^{\bar{u}_n}|x|^{2N}e^{\zeta} e^{w_n})dx=\pi(\sum_{i=1}^mb_i)^2-4\pi N\sum_{i=1}^mb_i+o(1).
\end{align*}
Thus, we have the following relation
\begin{equation*}
\sum_{i=1}^mb_i^2=\left(\sum_{i=1}^mb_i\right)^2-4N\sum_{i=1}^mb_i.
\end{equation*}

We finished the proof of Theorem \ref{thm:main-1}.

\end{proof}

\

\begin{proof}[\textbf{Proof of main Theorem \ref{thm:main-2}.}]

Since blow-up point $0$ is not a vortex point, i.e. $N=0$, by Theorem \ref{thm:1}, we know that $c=0$. By Lemma \ref{lem:pohozaev}, we see that there is only one kind of Pohozaev type identity in this case (i.e. local and global type Pohozaev identity are the same).  If there are more than two bubbles, denoting their energies by $b_i,\ i=1,...,m$, then by Theorem \ref{thm:main-1}, we get $$\sum_{i=1}^mb_i^2=(\sum_{i=1}^mb_i)^2,$$ which is a contradiction for $m\geq 2$ and $b_i\geq 4,\ i=1,...,m$. Thus there must be one bubble. We proved Theorem \ref{thm:main-2}.

\end{proof}

\

\begin{proof}[\textbf{Proof of main Theorem \ref{thm:main}.}] Conclusions of Theorem \ref{thm:main} is a direct consequence of Theorem \ref{thm:main-1} and Theorem \ref{thm:main-2}.
\end{proof}

\

\begin{proof}[\textbf{Proof of main Theorem \ref{thm:main-3}.}] If the vortex point $p_i$ is a blow-up point, then Theorem \ref{thm:main} tells us that the local mass $$\lim_{\rho\to 0}\lim_{n\to\infty}\alpha_n(\rho; p_i)\geq 4(1+N_i).$$ Since the total energy is $2 \bar{N}$ (divided by $2\pi$), thus we have $$2(1+N_i)\leq \bar{N},$$ which implies $\bar{N}\geq 4$.

%Thus, for $\bar{N}=2,3$, then $S$ contains only one non-vortex point. Moreover, the energy of bubble is $4$ and $6$ with respect to $\bar{N}=2$ and $\bar{N}=3$, which implies $\delta=0$ and $\delta>0$ in two cases.
%
%For $\bar{N}=4$, if $q_1$ is a vortex point, from inequality $$2(1+N_1)\leq \bar{N},$$ we have $N_1=1$ and $S=\{q_1\}$. If $(3-1)$ of Theorem \ref{thm:main} happens, we get the energy of bubble is $4(1+N_1)$ which implies that the constant $\delta=0$ in \eqref{equ:04}. If $(3-2)$ of Theorem \ref{thm:main} happens, then we know there are two bubbles whose energies satisfy $b_1=b_2=4$, which implies the bubbles satisfy Liouville equations.

For $\bar{N}=5$, if $q_1$ is a vortex point, from inequality $$2(1+N_1)\leq \bar{N},$$ we have $N_1=1$ and $S=\{q_1\}$. We claim the simple blow-up property holds, i.e. $(3-2)$ of Theorem \ref{thm:main} will not happen. Otherwise, in this case, there are just two bubbles with energies $b_i\geq 4,\ i=1,2$ satisfying $b_1+b_2=2\bar{N}=10$ and $b_1^2+b_2^2=(b_1+b_2)^2-4(b_1+b_2)$. Solving directly, we get $$b_i=5\pm\sqrt{5},\ i=1,2,$$ which is a contradiction that $5-\sqrt{5}<4$. Thus, the simple blow-up property holds.

For $\bar{N}=6$, if $S=\{q_1\}$ and $q_1$ is a vortex point with multiplicity $N_1=1$, since the total energy is $2\bar{N}=12$, if there are more than two bubbles, then we have the following two cases. $(i)$ if there are three bubbles, then there must be $b_1=b_2=b_3=4$. This is a contradiction to equality $b_1^2+b_2^2+b_3^2=(b_1+b_2+b_3)^2-4(b_1+b_2+b_3)$.  $(ii)$ if there are two bubbles, then we have $b_1+b_2=2\bar{N}=12$ and $b_1^2+b_2^2=(b_1+b_2)^2-4(b_1+b_2)=96$ which implies $$b_i=6\pm 2\sqrt{3},\ i=1,2.$$ This is also a contradiction since $6-2\sqrt{3}<4$.

For $\bar{N}=7$, if $q_1$ is a vortex point, from inequality $$2(1+N_1)\leq \bar{N},$$ we have following the two cases. $(1)\ N_1=1.$ In this case,  since the total energy is $2\bar{N}=14$, if there are more than two bubbles, then we have following two cases. $(1-i)$ If there are three bubbles whose energies denoted by  $b_1,\ b_2,\ b_3$. Then we have $b_1+b_2+b_3=14$ and $b_1^2+b_2^2+b_3^2=(b_1+b_2+b_3)^2-4(b_1+b_2+b_3)=140$. One can check by computer that there is no solution $(b_1,b_2,b_3)$ such that $b_i\geq 4,\ i=1,2,3$, which is a contradiction.  $(1-ii)$ If there are two bubbles, then we have $b_1+b_2=2\bar{N}=14$ and $b_1^2+b_2^2=(b_1+b_2)^2-4(b_1+b_2)=140$ which implies $$b_i=7\pm \sqrt{21},\ i=1,2.$$ This is also a contradiction since $7-\sqrt{21}<4$.

$(2)\ N_1=2$. Similarly to first case, we have the following two cases. $(2-i)$ If there are three bubbles whose energies denoted by  $b_1,\ b_2,\ b_3$. Then we have $b_1+b_2+b_3=14$ and $b_1^2+b_2^2+b_3^2=(b_1+b_2+b_3)^2-8(b_1+b_2+b_3)=84$. One can check by computer that there is no solution $(b_1,b_2,b_3)$ such that $b_i\geq 4,\ i=1,2,3$, which is a contradiction.  $(2-ii)$ If there are two bubbles, then we have $b_1+b_2=2\bar{N}=14$ and $b_1^2+b_2^2=(b_1+b_2)^2-8(b_1+b_2)=84$ which  is also a contradiction since $196=(b_1+b_2)^2\leq 2(b_1^2+b_2^2)\leq 168$.

We finished all the proof.
\end{proof}

\

% ----------------------------------------------------------------
%\bibliographystyle{amsplain}
%\bibliography{cankaowenxian}

\providecommand{\bysame}{\leavevmode\hbox to3em{\hrulefill}\thinspace}
\providecommand{\MR}{\relax\ifhmode\unskip\space\fi MR }
\providecommand{\MRhref}[2]{%
  \href{http://www.ams.org/mathscinet-getitem?mr=#1}{#2}
}
\providecommand{\href}[2]{#2}

\end{document}